\pdfoutput=1
\RequirePackage{ifpdf}
\ifpdf % We~are running pdfTeX in pdf mode
\documentclass[pdftex]{sigma}
\else
\documentclass{sigma}
\fi

\numberwithin{equation}{section}

\newtheorem{Theorem}{Theorem}[section]
\newtheorem*{Theorem*}{Theorem}
\newtheorem{Corollary}[Theorem]{Corollary}
\newtheorem{Lemma}[Theorem]{Lemma}
\newtheorem{Proposition}[Theorem]{Proposition}
\newtheorem{Fact}[Theorem]{Fact}
\newtheorem{Conjecture}[Theorem]{Conjecture}

\theoremstyle{definition}
\newtheorem{Definition}[Theorem]{Definition}

\newtheorem{Example}[Theorem]{Example}
\newtheorem{Remark}[Theorem]{Remark}
\newtheorem{Notation}[Theorem]{Notation}

\usepackage{bm}
\usepackage{verbatim}
\usepackage{cleveref}
\usepackage{adjustbox}

\renewcommand{\P}{\mathcal{P}}

\newcommand{\R}{\mathbb{R}}

\newcommand{\M}{\mathcal{M}}

\newcommand{\N}{\mathbb{N}}

\newcommand{\mult}{\mathrm{mult}}

\newcommand{\mfp}{\mathfrak{p}}
\newcommand{\mfq}{\mathfrak{q}}

\newcommand{\pols}{\mathcal P_n}

\newcommand{\cc}{\mathbb{C}}

\newcommand{\nn}{\mathbb{N}}

\newcommand{\rr}{\mathbb{R}}

\renewcommand{\tt}{\mathbb{T}}

\newcommand{\zz}{\mathbb{Z}}

\newcommand{\MM}{\mathcal{M}}

\newcommand{\PP}{\mathcal{P}}

\newmuskip\pFqmuskip

\newcommand*\pFqN[6][8]{%
 \begingroup % only local assignments
 \pFqmuskip=#1mu\relax
 \mathcode` =\string"8000
 \begingroup\lccode`\~=`
 \lowercase{\endgroup\let~}\pFqcomma
 {}_{#2}F_{#3}{\left(\genfrac..{0pt}{}{#4}{#5};#6\right)}%
 \endgroup
}
\newcommand{\pFqcomma}{\mskip\pFqmuskip}

\newcommand*\HGF[5]{%
 {\ }_{#1} F_{#2}{\left(\genfrac..{0pt}{}{#3}{#4};#5\right)}%
}

\newcommand*\HGP[3]{%
 \mathcal{H}_{#1}{\Bigl[\genfrac..{0pt}{1}{#2}{#3}\Bigl]}%
}

\newcommand*\HGPS[3]{%
 \mathcal{H}^{E}_{#1}{\Bigl[\genfrac..{0pt}{1}{#2}{#3}\Bigl]}%
}

\newcommand*\SRM[2]{%
 \rho{\Bigl[\genfrac..{0pt}{1}{#1}{#2}\Bigl]}%
}

\newcommand*\SRMS[2]{%
 \rho^{E}{\Bigl[\genfrac..{0pt}{1}{#1}{#2}\Bigl]}%
}

\newcommand{\bE}{\mathbb{E}}
\newcommand{\sfe}{\mathsf{e}}
\newcommand{\wtilde}{\widetilde}
\newcommand{\la}{\langle}
\newcommand{\ra}{\rangle}

\newcommand{\raising}[2]{\left(#1\right)^{\overline{#2}}}
\newcommand{\falling}[2]{\left(#1\right)^{\underline{#2}}}

\newcommand{\lift}{\bm{S}}
\newcommand{\halfeven}{\bm{Q}}

\newcommand{\dil}[1]{\operatorname{Dil}_{#1}}

\begin{document}

\allowdisplaybreaks

\newcommand{\arXivNumber}{2502.00254}

\renewcommand{\PaperNumber}{108}

\FirstPageHeading

\ShortArticleName{Even Hypergeometric Polynomials and Finite Free Commutators}

\ArticleName{Even Hypergeometric Polynomials\\ and Finite Free Commutators}

\Author{Jacob CAMPBELL~$^{\rm a}$, Rafael MORALES~$^{\rm b}$ and Daniel PERALES~$^{\rm c}$}

\AuthorNameForHeading{J.~Campbell, R.~Morales and D.~Perales}

\Address{$^{\rm a)}$~Department of Mathematics, University of Virginia, VA, USA}
\EmailD{\mail{cgh6uv@virginia.edu}}

\Address{$^{\rm b)}$~Department of Mathematics, Baylor University, TX, USA}
\EmailD{\mail{rafael\_morales2@baylor.edu}}

\Address{$^{\rm c)}$~Department of Mathematics, University of Notre Dame, IN, USA}
\EmailD{\mail{dperale2@nd.edu}}

\ArticleDates{Received April 15, 2025, in final form December 06, 2025; Published online December 21, 2025}

\Abstract{We study in detail the class of even polynomials and their behavior with respect to finite free convolutions. To this end, we use some specific hypergeometric polynomials and a variation of the rectangular finite free convolution to understand even real-rooted polynomials in terms of positive-rooted polynomials. Then, we study some classes of even polynomials that are of interest in finite free probability, such as even hypergeometric polynomials, symmetrizations, and finite free commutators. Specifically, we provide many new examples of these objects, involving classical families of special polynomials (such as Laguerre, Hermite, and Jacobi). Finally, we relate the limiting root distributions of sequences of even polynomials with the corresponding symmetric measures that arise in free probability.\looseness=-1}


\Classification{46L54; 33C20; 33C45}

\section{Introduction}

In the past decade, the subject of finite free probability has grown considerably due to its connections with geometry of polynomials, combinatorics, random matrix theory, and free probability. The core objects of study in finite free probability are polynomials of a fixed degree~$n$, and some convolution operations on such polynomials. These convolutions, called additive and multiplicative finite free convolutions and denoted by $\boxplus_n$ and $\boxtimes_n$, were studied a century ago \cite{szego1922bemerkungen,walsh1922location} but were recently rediscovered \cite{MSS15} as expected characteristic polynomials of certain random matrices. The finite free convolutions preserve various real-rootedness and interlacing properties, and when $n \to \infty$, they approximate the additive and multiplicative free convolutions of Voiculescu~\cite{voiculescu1992free}, which are denoted by $\boxplus$ and $\boxtimes$.

Since \cite{MSS15} first appeared in 2015, there have been several developments in finite free probability which expand on the parallel with free probability. In \cite{marcus}, Marcus developed finite $R$- and $S$-transforms in analogy with the corresponding analytic functions in free probability. On the other hand, the combinatorial side of finite free probability was developed by Arizmendi, Garza-Vargas, and Perales \cite{arizmendi2021finite,arizmendi2018cumulants} with a cumulant sequence for finite free convolutions, in analogy with the work of Nica and Speicher on free cumulants \cite{NS-book}.

There are two other operations on measures in free probability which have finite analogues that are relevant to this paper. One of them is the rectangular additive convolution studied by Gribinski and Marcus \cite{gribinski2022rectangular}, which is analogous to the rectangular free convolution defined by Benaych-Georges \cite{benaych2009rectangular}. The other is the finite free commutator operation studied by Campbell~\cite{campbell2022commutators}, which is analogous to the free commutator operation studied by Nica and Speicher \cite{NS98}. A~common thread between these two operations is their relation to \emph{even} polynomials: these are the polynomials whose roots come in positive-negative pairs.

Thus far, the class of even polynomials has not been studied in general in the context of finite free probability. In this paper, we carry out a detailed study of the class of even polynomials and compile some basic but useful results concerning their behavior with respect to finite free convolutions. In many cases, convolutions of even real-rooted polynomials can be reinterpreted as convolutions of polynomials with all non-negative roots and some specific hypergeometric polynomials that were studied recently in \cite{mfmp2024hypergeometric}. We should also mention that while working on this project, we became aware of a recent preprint \cite{campbell2024universality} that studies even entire functions in connection to finite free probability.

The basic tool in this paper is the map that takes an even polynomial, i.e., a polynomial of the form $\tilde{q}(x)=q\bigl(x^2\bigr)$, and returns the polynomial $q(x)$. We call this map $\halfeven_{m}$, where $m$ is the degree of $q$. An important property of this map is that $\tilde{q}(x)$ is real-rooted if and only if $q(x)$ is positive-rooted; this is helpful because the behavior of positive-rooted polynomials with respect to finite free convolutions is well understood. However, $\halfeven_{m}$ does not quite preserve finite free convolutions
\[\halfeven_{m}(p \boxplus_{2m} q) \neq \halfeven_{m}(p) \boxplus_m \halfeven_{m}(q)\qquad \text{and} \qquad\halfeven_{m}(p \boxtimes_{2m} q) \neq \halfeven_{m}(p) \boxtimes_m \halfeven_{m}(q).\]

To make these into equalities, we modify the right-hand sides using some particular hypergeometric polynomials of the kind studied in \cite{mfmp2024hypergeometric}. In the additive case, the result is related to a~generalization of the rectangular convolution of \cite{gribinski2022rectangular}. This generalization was first mentioned in a talk by Gribinski \cite{TalkGribinski}, and was recently studied in \cite{campbell2024free}. We also study an interesting operation in finite free probability which takes a single polynomial $p$ and returns the even polynomial~${p(x) \boxplus_n p(-x)}$, which we call the \emph{symmetrization} of $p$. This operation appeared naturally in \cite{campbell2022commutators} in relation to the finite free commutator.

Then, we establish even versions of the algebraic results of \cite{mfmp2024hypergeometric}, which concern the behavior of even hypergeometric polynomials with respect to finite free convolutions. In particular, we use some known product identities concerning hypergeometric series to provide many non-trivial examples of symmetrizations of hypergeometric polynomials (such as Laguerre, Hermite, or Jacobi) which are of interest in finite free probability.

We also use our framework of even hypergeometric polynomials to provide some new insight into finite free commutators: the result of \cite{campbell2022commutators} can be phrased in terms of even hypergeometric polynomials, and we provide some partial results concerning real-rootedness. We work out many examples of finite free commutators, and connect them asymptotically with known examples of commutators in free probability.

Finally, we study the asymptotic behavior of even polynomials in connection to free probability. In some cases we obtain new results in free probability. For instance, in the limit, the symmetrization operation tends to the analogous operation on probability measures, which we call free symmetrization: given a probability measure $\mu$, its free symmetrization is the measure~${\mu \boxplus \wtilde{\mu}}$, where $\wtilde{\mu}$ is just the pushforward of $\mu$ by $x\mapsto -x$. To the best of our knowledge this type of operation has not been studied systematically in free probability, but it has appeared sporadically in different contexts (see Remark~\ref{rem:free.sym}). Our machinery allows us to compute the free symmetrizations of some special distributions, including the Marchenko--Pastur, reversed Marchenko--Pastur, and free beta distributions.

Besides this introductory section, the rest of the paper is organized as follows. In Section~\ref{sec:prelim}, we review some preliminaries on polynomials, measures, and finite free probability. In Section~\ref{sec:evenpol}, we establish our basic framework for even polynomials and study its behavior in relation to finite free convolutions. We specialize to the study of even hypergeometric polynomials in Section~\ref{sec:EvenHyper}. In~Section~\ref{sec:commutator}, we use even hypergeometric polynomials to study the finite free commutator operation and provide many examples. In Section~\ref{sec:asymptotics}, we study the asymptotic behavior of families of even polynomials and relate it to results in free probability. See Section~\ref{sec:results.standard.notation} for a~summary of the main results using standard hypergeometric series notation ${}_iF_j$.

%%%%%%%%%%%%%%%%%%%%%%%%%%%%%%%%%%%%%%%%%%%%%%%%%%%%
\section{Preliminaries}
\label{sec:prelim}
%%%%%%%%%%%%%%%%%%%%%%%%%%%%%%%%%%%%%%%%%%%%%%%%%

%%%%%%%%%%%%%%%%%%%%%%%%%%%%%%%%%%%%%%%%%%%%%%%%%%%%%%
\subsection{Polynomials and their coefficients} \label{subsec:polynomials}

We start by introducing some notation. We denote by $\P_n$ the set of monic polynomials (over the complex plane $\cc$) of degree $n$. To specify that all the roots of a polynomial belong to a~specific region $K \subseteq \cc$, we use the notation $\P_n(K)$. For most of our results, $K$ is going to be either the set of real numbers $\rr$, the set positive real number $\rr_{>0}$, or the set of negative real numbers~$\rr_{<0}$.

\begin{Notation}[roots and coefficients]
 Given a polynomial $p \in \P_n$, we denote its roots by $\lambda_1(p), \dots, \lambda_n(p)$. Every polynomial $p \in \P_n$ can be written in the form
 \begin{equation}\label{monicP}
 p(x)=\sum_{k=0}^n x^{n-k} (-1)^k \sfe_k(p)
 \end{equation}
 for some coefficients $\sfe_k(p)$. Here, $\sfe_0(p)=1$ as $p$ is monic. There is a specific formula for these coefficients
 \[ \sfe_k(p) = \sum_{1 \leq i_1 < \dots < i_k \leq n} \lambda_{i_1}(p) \cdots \lambda_{i_k}(p) \]
 for $0 \leq k \leq n$. These are the so-called \emph{elementary symmetric polynomials} in the roots of $p$.
\end{Notation}

\begin{Notation}[dilation]
 The \emph{dilation} of a polynomial $p$ by a non-zero scalar $\alpha$ is defined as ${ [\dil{\alpha} p](x) = \alpha^{n} p\bigl(\alpha^{-1} x\bigr)}$.
 The roots of $\dil{\alpha} p$ are the roots of $p$ scaled by $\alpha$: $\alpha \lambda_1(p), \dots, \allowbreak\alpha \lambda_n(p)$.
\end{Notation}

\subsection{Measures and asymptotic empirical root distribution}
\label{sec:measures}

We will be interested in sequences of polynomials with increasing degree, whose zero distribution tend in the limit to a probability measure. We denote by $\M$ the set of probability measures on the complex plane. Similar to our notation for polynomials, for $K \subseteq \cc$, we denote by $\M(K)$ the set of probability measures supported on $K$. In particular, $\M(\rr)$ is the set of probability measures supported on the real line. We also denote by $\M^E(\rr)$ the set of symmetric probability measures on $\rr$.

\begin{Notation}[Cauchy transform]
 For a probability measure $\mu \in \M(\rr)$, the \emph{Cauchy transform} of $\mu$ is defined by
 \[ G_{\mu}(z) := \int_{\rr} \frac{1}{z-t} {\rm d}\mu(t). \]
 This is an analytic function from the upper half-plane to the lower half-plane. Among other things, the Cauchy transform encodes weak convergence: for a sequence $(\mu_n)_{n \geq 1}$ in $\M(\rr)$ and a measure $\mu \in \M(\rr)$, we have $\mu_n \to \mu$ weakly if and only if $G_{\mu_n}(z) \to G_{\mu}(z)$ pointwise. This result can be found, e.g., in \cite[Remark 12, Theorem 13, Lemma~3 of Section~3.1]{MS-book}.
\end{Notation}

\begin{Notation} \label{nota:Q.measures}
 For $\mu \in \M^E(\rr)$, let $Q(\mu)$ be the pushforward of $\mu$ along the map $\rr \to \rr_{\geq 0} \colon\allowbreak {t \mapsto t^2}$. This is a probability measure supported on $\rr_{\geq 0}$, and it has appeared before in the free probability literature \cite{arizmendi2009S, NS98}. A useful description of $Q(\mu)$ was given in \cite[Proposition~5]{arizmendi2009S} in terms of the Cauchy transform:
$ G_{\mu}(z) = z G_{Q(\mu)}\bigl(z^2\bigr)$.
 In the other direction, for $\nu \in \M(\rr_{\geq 0})$, write
$S(\nu) = \frac{1}{2}(S_+(\nu)+S_-(\nu)) $,
 where $S_+(\nu)$ and $S_-(\nu)$ are the pushforwards of $\nu$ along the maps
$\rr_{\geq 0} \to \rr_{\geq 0} \colon t \mapsto \sqrt{t}$ and $\rr_{\geq 0} \to \rr_{\leq 0} \colon t \mapsto -\sqrt{t}$,
 respectively.
\end{Notation}

\begin{Lemma}\label{lem:Q.homeomorphism}
 $Q \colon \M^E(\rr) \to \M(\rr_{\geq 0})$ is a homeomorphism with respect to weak convergence, with inverse $S$.
\end{Lemma}
\begin{proof}
 The map $Q$ is weakly continuous because it is a pushforward along a continuous map~\cite[Theorem 2.7]{Billingsley-convergence}. One can check that $S$ is the inverse of $Q$ by integrating against bounded continuous functions, and $S$ is weakly continuous for the same reason as $Q$.
\end{proof}

In this paper, we will often be concerned with root distributions of polynomials. For a~non-zero polynomial $p$, the \emph{empirical root distribution} (or \emph{zero counting measure}) of $p$ is the measure
 \[ \rho(p) := \frac{1}{\deg(p)} \sum_{\alpha \text{ root of } p} \delta_{\alpha} \in \M ,\]
 where the roots in the sum are counted with multiplicity and $\delta_\alpha$ is the Dirac delta (unit mass) placed at the point $\alpha$.

\begin{Definition}
 \label{def:converging}
 We say that a sequence $\mathfrak p= \left(p_{n_k}\right)_{k\geq1}$ of polynomials is \emph{converging} if
 \begin{itemize}\itemsep=0pt
 \item $(n_k)_{k \geq 1}$ is a strictly increasing subsequence of integers,
 \item $p_{n_k}$ has degree $n_k$ for $k \geq 1$, and
 \item there is a measure, denoted by $\rho(\mfp)$, such that $\rho(p_{n_k}) \to \rho(\mfp)$ weakly as $k \to \infty$.
 \end{itemize}
 For our purposes, the subsequence $(n_k)_{k \geq 1}$ will usually be the full sequence of integers or the subsequence of even integers.
\end{Definition}

\subsection{Free convolution}

In this paper, we will occasionally use some tools from free probability. In this section we will briefly review the facts we need, following the references \cite{MS-book,NS-book}.

Given a compactly supported measure $\mu \in \M(\rr)$ with moment sequence $(m_k)_{k \geq 0}$, the \emph{moment generating function} of $\mu$ is given by
\[ M_{\mu}(z) = \sum_{k=1}^{\infty} m_k z^k , \]
and the Cauchy transform $G_{\mu}$ has the Laurent expansion
\[ G_{\mu}(z) = \sum_{k=0}^{\infty} m_k z^{-(k+1)} \]
in a neighborhood of $\infty$. This yields the relation $M_{\mu}(z) = z^{-1} G_{\mu}\bigl(z^{-1}\bigr) - 1$.

The \emph{$R$-transform} and \emph{$S$-transform} of $\mu$ are defined by
\[ R_{\mu}(z) = G_{\mu}^{\la -1 \ra}(z) - \frac{1}{z} \qquad \text{and} \qquad S_{\mu}(z) = \frac{z+1}{z} M_{\mu}^{\la -1 \ra}(z). \]
Here, $\langle -1 \rangle$ denotes the compositional inverse, whose existence and analyticity around the origin for \smash{$G_{\mu}^{\langle -1 \rangle}(z)$} and \smash{$M_{\mu}^{\langle -1 \rangle}(z)$} are guaranteed provided that $m_0 = 1$ and $m_1 \neq 0$. The free additive and multiplicative convolutions correspond to the sum and product of free random variables; alternatively one can compute them using the \emph{$R$-transform} and \emph{$S$-transform}, see \cite[Definition~3.47 and Theorem 4.23]{MS-book}.

\begin{Definition}
 Let $\mu,\nu \in \M(\rr)$ be compactly supported. Then
 \begin{enumerate}\itemsep=0pt
 	\item[(1)] their \emph{free additive convolution} is the unique measure $\mu \boxplus \nu$ satisfying
$R_{\mu \boxplus \nu}(z) = R_{\mu}(z) + R_{\nu}(z)$;
 \item[(2)] if $\mu,\nu \in \M(\rr_{\geq 0})$, their \emph{free multiplicative convolution} is the unique measure $\mu \boxtimes \nu \in \M(\rr_{\geq 0})$ satisfying
$
 S_{\mu \boxtimes \nu}(z) = S_{\mu}(z) S_{\nu}(z)$.
 \end{enumerate}
\end{Definition}

\subsection{Finite free convolution of polynomials}

In this subsection, we summarize some definitions and results on the finite free additive and multiplicative convolutions that will be used throughout this paper. First, let us establish some notation for rising and falling factorials.

\begin{Notation}
 For $a \in \cc$ and $k \in \zz_{\geq 0}:=\N\cup\{0\}$, the \emph{rising} and \emph{falling factorials}\footnote{The rising and falling factorials are both sometimes called the \emph{Pochhammer symbol}. The notation $(x)_n$ is sometimes used to refer to either the rising or falling factorial; we prefer the clear notation laid out in the text.} are~respectively defined as
 \smash{$ \raising{a}{k} := a(a+1) \cdots (a+k-1) $}
 and
$ \smash{\falling{a}{k} }:= a(a-1) \cdots (a-k+1)= \smash{\raising{a-k+1}{k}}$.
 The following relations follow from the definition and will be useful later. For $a \in \cc$ and $j,k \in \zz_{\geq 0}$ with $j \leq k$, we have
 \begin{gather}% \label{connectionRandL} \label{connectionRandL2}
 \falling{a}{k} = (-1)^k \raising{-a}{k} ,\qquad
 \raising{a}{k} = \raising{a}{k-j} \falling{a+k-1}{j}, \qquad\text{and} \nonumber\\
 \falling{2a}{2k} = 2^{2k} \falling{a}{k} \falling{a-\frac{1}{2}}{k}. \label{eq.falling.factorial.2m}
 \end{gather}
\end{Notation}

We are now ready to introduce the multiplicative and additive convolutions, as defined in~\cite[Definitions~1.1 and~1.4]{MSS15}.

\begin{Definition}[additive and multiplicative convolutions] \label{def:FinFreeConv}
 Consider polynomials $p,q \in \P_n$. We define the \emph{finite free additive convolution} of $p$ and $q$ as the polynomial $p \boxplus_n q\in \P_n$ with coefficients given by
 \[ %\label{eq:def.box.plus}
 \sfe_k(p \boxplus_n q) = \falling{n}{k} \sum_{i+j=k} \frac{ \sfe_i(p) \sfe_j(q)}{\falling{n}{i} \falling{n}{j}} \qquad \text{for}\quad 0 \leq k \leq n.
 \]
 We define the \emph{finite free multiplicative convolution} of $p$ and $q$ as the polynomial $p \boxtimes_n q\in \P_n$ with coefficients given by
 \[% \label{eq:def.box.times}
 \sfe_k(p \boxtimes_n q) = \frac{1}{\binom{n}{k}} \sfe_k(p) \sfe_k(q) \qquad \text{for}\quad 0 \leq k \leq n.
 \]
\end{Definition}

\begin{Remark}
 \label{rem:additive.conv.diff.form}
 We will also use an equivalent description of $\boxplus_n$ in terms of certain differential operators, which was also given in \cite[equation~(2)]{MSS15}. Namely, if we can write $p(x) = P\left(\frac{\partial}{\partial x}\right) x^n$ and $q(x) = Q\left(\frac{\partial}{\partial x}\right) x^n$ for some polynomial $P$ and $Q$, then
 \begin{equation} \label{thirdDefAdditiveConv}
 p(x) \boxplus_n q(x) = P\left(\frac{\partial}{\partial x}\right)Q\left(\frac{\partial}{\partial x}\right)x^n.
 \end{equation}
 Notice that for a given $p$, there is a unique polynomial $P$ of degree $n$ satisfying $p(x) = P\left(\frac{\partial}{\partial x}\right) x^n$, and is explicitly given by
 \[%\label{diffOperatorforP}
 P\left(\frac{\partial}{\partial x}\right) = \sum_{k=0}^n (-1)^k \frac{\sfe_k(p)}{\falling{n}{k}} \left(\frac{\partial}{\partial x}\right)^k.
 \]
 To construct other differential operators that yield~$p$, one can add terms of the form $\left(\frac{\partial}{\partial x}\right)^k$ with~${k>n}$, which vanish when applied to~$x^n$.
\end{Remark}

\begin{Remark}[basic properties] \label{rem.convolution.basic.facts}
 Directly from the definition, one can derive some basic properties of the binary operations $\boxplus_n$ and $\boxtimes_n$. For example, they are bilinear, associative, and commutative. Another property of $\boxtimes_n$ is related to scaling: for $\alpha \in \cc$ and $p \in \P_n$, we have
$ p(x) \boxtimes_n (x-\alpha)^n = \dil{\alpha} p$.
\end{Remark}

The main property of finite free convolutions is that they preserve real-rootedness, in the following sense.

\begin{Theorem}[\cite{szego1922bemerkungen,walsh1922location}] \label{thm:realrootedness}
 Let $p, q \in \P_n(\R)$. Then
 \begin{enumerate}\itemsep=0pt
 \item[$(1)$] $p \boxplus_n q \in \P_n(\rr)$;
 \item[$(2)$] if either $p \in \P_n(\rr_{\geq 0})$ or $q \in \P_n(\rr_{\geq 0})$, then $p \boxtimes_n q \in \P_n(\rr)$;
 \item[$(3)$] if both $p,q \in \P_n(\rr_{\geq 0})$, then $p \boxtimes_n q \in \P_n(\rr_{\geq 0})$.
 \end{enumerate}
\end{Theorem}

\begin{Remark}
 \label{rem:ruleofsigns}
 It is easy to extend (3) in Theorem~\ref{thm:realrootedness} to a ``rule of signs" for the behavior of roots under the operation $\boxtimes_n$. Specifically, since $\dil{-1} p = p \boxtimes_n (x+1)^n$ by Remark~\ref{rem.convolution.basic.facts}, we have the following:
 \begin{itemize}\itemsep=0pt
 \item if $p,q \in \P_n(\rr_{\leq 0})$, then $p \boxtimes_n q \in \P_n(\rr_{\geq 0})$;
 \item if $p \in \P_n(\rr_{\leq 0})$ and $q \in \P_n(\rr_{\geq 0})$, then $p \boxtimes_n q \in \P_n(\rr_{\leq 0})$.
 \end{itemize}
\end{Remark}

Moreover, the good behavior of the roots under these convolutions goes beyond the real line. From the work of Walsh \cite{walsh1922location} and Szeg\H{o} \cite{szego1922bemerkungen}, one can also bound the location of the roots in the complex plane of $p\boxtimes_n q$ from the location of the roots of $p$ and $q$, see also \cite[Theorems 16.1 and 18.1]{marden1966geometry}. Using the result by Szeg\H{o}, one can show that the set of polynomials with all the roots in the unit circle $\tt:=\{z\in \cc\mid |z|=1\}$ is closed under multiplicative convolution.

\begin{Lemma}[$\P_{n}(\tt)$ is closed under multiplicative convolution]
\label{lem:multiplicative.unit.circle}
Let $n\in\nn$ and assume $p,q\in\P_{n}(\tt)$, then $p\boxtimes_n q\in\P_{n}(\tt)$.
\end{Lemma}

This result is the particular case $\theta=0$ of \cite[Theorem 5]{suffridge1976starlike}. For the reader's convenience, we will provide a direct proof using Szeg\H{o}'s theorem (see \cite[Theorem~16.1]{marden1966geometry}).

\begin{proof}
Recall from \cite[Section 12]{marden1966geometry} that circular regions include the closed interior or exterior of circles. Since $p\in\P_{n}(\tt)$, in particular the zeros of $p$ lie in the circular region $D:=\{z\in\cc\mid |z|\leq 1\}$. By \cite[Theorem~16.1]{marden1966geometry},\footnote{Note that in Szeg\H{o}'s theorem, $h$ is the multiplicative convolution of $f$ and $g$ up to a dilation by $-1$, namely~${h(-z)=f(z)\boxtimes_n g(z)}$} we can conclude that all the zeros of $[p\boxtimes_n q](-x)$ lie in $-D$, meaning that $p\boxtimes_n q\in \P_{n}(D)$. Similarly, since $p\in\P_{n}(\tt)$, then all the zeros of $p$ lie in the circular region $E:=\{z\in\cc\mid |z|\geq 1\}$. Using \cite[Theorem 16.1]{marden1966geometry} again, we obtain $p\boxtimes_n q\in \P_{n}(E)$.
Since~${D\cap E= \tt}$ we conclude that $p\boxtimes_n q\in \P_{n}(\tt)$.
\end{proof}

The connection between finite free probability and free probability is revealed in the asymptotic regime; this was first observed by Marcus \cite[Section~4]{marcus} and formalized later using finite free cumulants \cite{arizmendi2021finite,arizmendi2018cumulants}.

\begin{Theorem}[{\cite[Corollary~5.5]{arizmendi2018cumulants}, \cite[Theorem 1.4]{arizmendi2021finite}}]
 \label{thm:finiteAsymptotics}
 Let $\mfp:=\left(p_n\right)_{n=1}^{\infty}$ and $\mfq:=\left(q_n\right)_{n=1}^{\infty}$ be two converging sequences of polynomials in the sense of Definition~{\rm\ref{def:converging}}.
 \begin{enumerate}\itemsep=0pt
 \item[$(i)$] If $\mfp,\mfq \subset \P(\rr)$, then $\left(p_n \boxplus_n q_n\right)_{n=1}^{\infty}$ weakly converges to $\rho(\mfp) \boxplus \rho(\mfq)$.
 \item[$(ii)$] If $\mfp \subset \P(\rr)$ and $\mfq\subset \P(\rr_{>0})$ then $\left(p_n \boxtimes_n q_n\right)_{n=1}^{\infty}$ weakly converges to $\rho(\mfp) \boxtimes \rho(\mfq)$.
 \end{enumerate}
\end{Theorem}

\subsection{Hypergeometric polynomials and examples}
\label{subsec:HGpolynomials}

A large class of polynomials with real roots is contained in the class of hypergeometric polynomials; these were recently studied in connection with finite free probability in \cite{mfmp2024hypergeometric}. This class contains several important families, such as Bessel, Laguerre and Jacobi polynomials. These polynomials -- and their specializations such as Hermite polynomials -- constitute a rich class of examples in the theory of finite free probability.

\begin{Definition}
 \label{def:H.polynomials}
 For $i,j \in \zz_{\geq 0}$ and $n \geq 1$, let
 \begin{equation} \label{eq:condition.hypergeometric}
 a_1, \dots, a_i \in \rr \setminus \left\{ 0,\frac{1}{n},\frac{2}{n}, \dots, \frac{n-1}{n}\right\} \qquad \text{and}\qquad b_1,\dots,b_j \in \rr.
 \end{equation}
 Define the polynomial \smash{$\mathcal{H}_{n} \bigl[\begin{smallmatrix} b_1,\dots,b_j \\ a_1,\dots,a_i\end{smallmatrix}\bigr]$} to be the unique monic polynomial of degree $n$ with coefficients in representation \eqref{monicP} given by
 \[ \sfe_k \left( \HGP{n}{b_1,\dots,b_j}{a_1,\dots,a_i} \right) := \binom{n}{k} \frac{\falling{nb_1}{k} \cdots \falling{nb_j}{k}}{\falling{na_1}{k}\cdots \falling{na_i}{k}} \]
 for $1 \leq k \leq n$.

 To simplify notation, for a tuple $\bm a =(a_1, \dots, a_i)$ and constants $n, m\in\R$, we will write
 \begin{align*}
\falling{\bm a}{k} := \prod_{s=1}^i \falling{a_s}{k}\qquad \text{and}\qquad
n \bm a +m := (na_1+m, \dots, na_i+m).
\end{align*}
 Then, for tuples $\bm a = (a_1,\dots,a_i)$ and $\bm b = (b_1,\dots,b_j)$ satisfying equation \eqref{eq:condition.hypergeometric}, the hypergeometric polynomial \smash{$\mathcal{H}_{n}\bigl[\begin{smallmatrix} \bm b\\ \bm a \end{smallmatrix}\bigr]$} has coefficients given by
 \[ \sfe_k\bigl( \HGP{n}{\bm b}{\bm a} \bigr) = \binom{n}{k} \frac{\falling{\bm b n}{k}}{\falling{\bm a n}{k}}= \binom{n}{k} \frac{\prod_{t=1}^j \falling{nb_t}{k}}{\prod_{s=1}^i \falling{na_s}{k}}\]
 for $1 \leq k \leq n$.
\end{Definition}

\begin{Remark}
 The reason we call \smash{$\mathcal{H}_{n}\bigl[\begin{smallmatrix} \bm b\\ \bm a \end{smallmatrix}\bigr]$} ``hypergeometric'' is that it can be identified as a~terminating \emph{generalized hypergeometric series} \cite{MR2656096, MR2723248}
 \[%\label{HpolynomialasF}
 \HGP{n}{\bm b}{\bm a}(x) = \frac{(-1)^n \falling{\bm b n}{n}}{\falling{\bm a n}{n}} \HGF{i+1}{j}{-n,\bm a n -n+1}{\bm b n -n+1}{x},
 \]
 where $\bm c n -n+1$ means that we multiply each entry of $\bm c$ by $n$ and then add $-n+1$, and we use the standard notation ${}_iF_j$ for a generalized hypergeometric series. Namely, for tuples~${\bm a = (a_1,\dots,a_i) \in \rr^i}$ and $\bm b = (b_1,\dots,b_j) \in \rr^j$, we write
 \[% \label{hypergeoSeriesBis}
 \HGF{i}{j}{\bm a}{\bm b}{x} := \sum_{k=0}^{\infty} \frac{\raising{\bm a}{k}}{\raising{\bm b}{k}} \frac{x^k}{k!}.
 \]
\end{Remark}

In \cite{mfmp2024hypergeometric}, it was noticed that these hypergeometric polynomials behave well with respect to the finite free convolutions from Definition~\ref{def:FinFreeConv}.

\begin{Theorem}[{\cite[equations (82)--(84)]{mfmp2024hypergeometric}}]
 \label{thm:MFMP}
 Consider tuples $\bm a_1$, $\bm a_2$, $\bm a_3$, $\bm b_1$, $\bm b_2$, and $\bm b_3$ with lengths $i_1$, $i_2$, $i_3$, $j_1$, $j_2$, and $j_3$, respectively. Then
 \begin{enumerate}\itemsep=0pt
 \item[$(1)$] the multiplicative convolution of hypergeometric polynomials is another hypergeometric polynomial
 \[ %\label{eq:multiplicative.hyper}
\HGP{n}{\bm b_1}{\bm a_1 } \boxtimes_n \HGP{n}{\bm b_2}{\bm a_2 }=\HGP{n}{\bm b_1, \bm b_2 }{\bm a_1, \bm a_2} ;
 \]
 \item[$(2)$] if we have the factorization
 \[ %\label{eq:product.hyper.functions}
\HGF{j_1}{i_1}{-n\bm b_1}{-n\bm a_1}{x} \HGF{j_2}{i_2}{-n\bm b_2}{-n\bm a_2}{ x}=\HGF{j_3}{i_3}{-n\bm b_3}{-n\bm a_3}{x}
 \]
 and write $s_l=(-1)^{i_l+j_l+1}$ for $l=1,2,3$, then it holds that
 \[% \label{eq:additive.hyper}
\HGP{n}{\bm b_1}{\bm a_1 }(s_1 x) \boxplus_n \HGP{n}{\bm b_2}{\bm a_2 }(s_2 x) = \HGP{n}{\bm b_3}{\bm a_3 }(s_3 x).
 \]
 \end{enumerate}
\end{Theorem}

The limiting distribution of hypergeometric polynomials can be expressed concretely in terms of the $S$-transform.

\begin{Theorem}[{\cite[Theorem 3.9]{martinez2025zeros}, \cite[Corollary 10.8]{arizmendi2024s}}]
 \label{thm:HypLimit}
 For integers $i,j\geq 0$, consider tuples $\bm A=(A_1,\dots, A_i) \in (\rr\setminus [0,1))^i$ and $\bm B=(B_1,\dots, B_j)\in (\rr\setminus \{0\})^j$. Assume that $\mfp=(p_n)_{n\geq 0}$ is a sequence of polynomials such that
 \[%\label{sequencePn}
 p_n=\dil{n^{i-j}}\HGP{n}{\bm b_n}{\bm a_n}\in\pols(\rr_{> 0}),
 \]
 where the tuples of parameters $\bm a_n \in \rr^i$ and $\bm b_n\in \rr^j$ have a limit given by
% \begin{equation} \label{assumptionlimits}
$ \lim_{n\to\infty} \bm a_n=\bm A$, and $\lim_{n\to\infty} \bm b_n=\bm B$.
 Then $\mfp$ is a converging sequence in the sense of Definition~{\rm\ref{def:converging}}. Moreover, $\rho(\mfp)\in \MM(\rr_{\geq 0})$ has $S$-transform given by
 \[% \label{Stranformrational}
 S_{\rho(\mfp)}(z)= \frac{\prod_{r=1}^i (z+A_r) }{\prod_{s=1}^j(z+B_s)}.
 \]
\end{Theorem}

Recall that since $\rho(\mfp)\in\MM(\rr_{\geq 0})$, the measure is determined by the $S$-transform. In view of the previous result, we will use the following notation.

\begin{Notation}
 \label{nota:rhoparamseven}
 If $\mu\in \MM(\rr_{\geq 0})$ is a measure with $S$-transform of the form
 \[ S_{\mu}(z)= \frac{\prod_{r=1}^i (z+a_r) }{\prod_{s=1}^j(z+b_s)} , \]
 for some parameters $a_1,\dots, a_i$ and $b_1,\dots,b_j$, then we say that $\mu$ is an \emph{$S$-rational measure} and denote it by \smash{$\rho \bigl[\begin{smallmatrix} b_1,\dots, b_j\\ a_1,\dots, a_i\end{smallmatrix}\bigr]$}.

 Moreover, we denote the square root of an $S$-rational measure by
 \[ \SRMS{b_1,\dots, b_j}{a_1,\dots, a_i} := S \left( \SRM{b_1,\dots, b_j}{a_1,\dots, a_i} \right), \]
 where $S$ is the bijection from Notation~\ref{nota:Q.measures}.
\end{Notation}

Notice that, if $a_r=b_s$ for some $r\leq i$ and $s\leq j$, then
\[% \label{eq:paramcancel}
 \SRM{b_1,\dots, b_j}{a_1,\dots, a_i}=\SRM{b_1,\cdots,b_{s-1},b_{s+1},\cdots, b_j}{a_1,\cdots,a_{r-1},a_{r+1},\cdots, a_i}.
 \]

To finish this section, we survey some particular cases of Theorem~\ref{thm:HypLimit}. Including how Laguerre tends to Marchenko--Pastur, Bessel tends to reversed Marchenko--Pastur, and Jacobi tends to Free beta.

\begin{Example}[identities]
 The simplest cases of Definition~\ref{def:H.polynomials} are the following:
 \[ \HGP{n}{\bm a}{\bm a}(x) = \HGP{n}{\cdot}{\cdot}(x) = (x-1)^n \qquad \text{and}\qquad \HGP{n}{0}{\bm a}(x) = x^n. \]
 These polynomials are the identities for the operations $\boxtimes_n$ and $\boxplus_n$, respectively.
\end{Example}

\begin{Example}[Laguerre and Hermite polynomials]
 \label{exa:LaguerreHermite}
 The analogues of the semicircular and free Poisson distributions in finite free probability are the Hermite and Laguerre polynomials respectively. These appear in finite free analogues of the central limit theorem and Poisson limit theorem \cite{marcus}.

 The Laguerre polynomials of interest in this paper can be written in terms of the hypergeometric polynomials \smash{$\mathcal{H}_{n}\bigl[\begin{smallmatrix} b\\ \cdot\end{smallmatrix}\bigr]$}. There are some ranges of the parameter $b$ which yield real-rooted polynomials:
 \begin{itemize}\itemsep=0pt
 \item if $b \in \big\{\frac{1}{n},\dots,\frac{n-1}{n} \big\}$, then \smash{$\mathcal{H}_{n}\bigl[\begin{smallmatrix} b\\ \cdot\end{smallmatrix}\bigr] \in \P(\rr_{\geq 0})$}, and $0$ is a root with multiplicity $(1-b)n$;
 \item if $b \in \bigl(\frac{n-2}{n},\frac{n-1}{n}\bigr)$, then \smash{$\mathcal{H}_{n}\bigl[\begin{smallmatrix} b\\ \cdot\end{smallmatrix}\bigr] \in \P(\rr)$};
 \item if $b > 1-\frac{1}{n}$, then \smash{$\mathcal{H}_{n}\bigl[\begin{smallmatrix} b\\ \cdot\end{smallmatrix}\bigr] \in \P(\rr_{> 0})$}.
 \end{itemize}
 See \cite[Table 1, (56)]{mfmp2024hypergeometric} for more details.

 We will be particularly interested in the following scaled version of a Laguerre polynomial, to which we give a special name: for $\lambda \geq 1$, write
\smash{$L_n^{(\lambda)} = \dil{\frac{1}{n}}
\mathcal{H}_{n}\bigl[\begin{smallmatrix} \lambda\\ \cdot\end{smallmatrix}\bigr]$}.
 This polynomial can be understood as the finite analogue of the free Poisson (also known as Marchenko--Pastur) distribution in free probability. Indeed, it appears as the limiting polynomial in the finite free Poisson limit theorem \cite{marcus} and it converges in empirical root distribution to the free Poisson distribution with rate $\lambda$:
 \[ \rho\bigl(L_n^{(\lambda)}\bigr) \to \frac{1}{2\pi} \frac{\sqrt{(r_+-x)(x-r_-)}}{x} \mathbf{1}_{[r_-,r_+]}(x) {\rm d}x \]
 weakly as $n \to \infty$, where $r_{\pm} := \lambda + 1 \pm 2\sqrt{\lambda}$, and $\mathbf{1}_{[a,b]}$ stands for the indicator function $\mathbf{1}_{[a,b]}(x)=1$ if $x\in[a,b]$ and 0 otherwise. This is a classical result; see \cite[Section 5.3]{mfmp2024hypergeometric} and its references.

 The Hermite polynomials used in this paper are defined as follows:
 \[ H_{2m}(x) := m^{-m} \HGP{m}{1-\frac{1}{2m}}{\cdot}\bigl(mx^2\bigr) = \dil{\sqrt{\frac{1}{m}}} \HGPS{m}{1-\frac{1}{2m}}{\cdot}. \]
 This polynomial is known to be the appropriate finite analogue of the semicircular distribution in free probability. It appears as the limiting polynomial in the finite free central limit theorem~\cite{marcus} and it converges in empirical root distribution to the semicircular distribution with radius~$2$:
 \[ \rho(H_{2m}) \to \frac{1}{2\pi} \sqrt{4-x^2} \mathbf{1}_{[-2,2]}(x) {\rm d}x \]
 weakly as $n \to \infty$. This is also a well-known classical result, see, e.g., \cite{KM16}.

 The resemblance between the definitions of $H$ and $L^{(\lambda)}$ is meaningful, and will be elaborated in Example \ref{exa:HermiteSquared}.
\end{Example}

\begin{Example}[Bessel polynomials]
 \label{exa:BesselRMP}
 The Bessel polynomials of interest in this paper can be written in terms of the hypergeometric polynomials \smash{$\mathcal{H}_{n}\bigl[\begin{smallmatrix} \cdot\\ a\end{smallmatrix}\bigr]$}. Some known results on their roots are the following:
 \begin{itemize}\itemsep=0pt
 \item if $a\in \bigl(0,\frac{1}{n}\bigr)$, then \smash{$\mathcal{H}_{n}\bigl[\begin{smallmatrix} \cdot\\ a\end{smallmatrix}\bigr]\in \P(\rr)$};
 \item if $a<0$, then \smash{$\mathcal{H}_{n}\bigl[\begin{smallmatrix} \cdot\\ a\end{smallmatrix}\bigr]\in \P(\rr_{< 0})$}.
 \end{itemize}
 For $a < 0$, denote
\smash{$ C_n^{(a)} := \dil{-n} \mathcal{H}_{n}\bigl[\begin{smallmatrix} \cdot\\ a\end{smallmatrix}\bigr] \in \P(\rr_{< 0})$}.
 The asymptotics of these Bessel polynomials are also known:
 \[ \rho\bigl(C_n^{(a)}\bigr) \to \frac{-a}{2\pi} \frac{\sqrt{(r_+-x)(x-r_-)}}{x^2} \mathbf{1}_{[r_-,r_+]}(x) {\rm d}x \]
 weakly as $n \to \infty$, where \smash{$r_{\pm} := -\frac{1}{a-2\pm 2\sqrt{1-a}}$}. See \cite[Section~5.3]{mfmp2024hypergeometric}.
\end{Example}

\begin{Example}[Jacobi polynomials]
 \label{exa:Jacobi}
 The Jacobi polynomials of interest in this paper can be written in terms of the hypergeometric polynomials \smash{$\mathcal{H}_{n}\bigl[\begin{smallmatrix} b\\ a\end{smallmatrix}\bigr]$}. Let us recall some simple combinations of parameters which produce real-rooted Jacobi polynomials:
 \begin{itemize}\itemsep=0pt
 \item \smash{$\mathcal{H}_{n}\bigl[\begin{smallmatrix} b\\ a\end{smallmatrix}\bigr]\in \P_n([0,1])$} when $b >1-\frac{1}{n}$ and $a>b+1-\frac{1}{n}$;
 \item \smash{$\mathcal{H}_{n}\bigl[\begin{smallmatrix} b\\ a\end{smallmatrix}\bigr]\in \P_n(\rr_{<0})$} when $b>1-\frac{1}{n}$ and $a<0$;
 \item \smash{$\mathcal{H}_{n}\bigl[\begin{smallmatrix} b\\ a\end{smallmatrix}\bigr]\in \P_n(\rr_{>0})$} when $a<0$ and $b<a-1+\frac{1}{n}$.
 \end{itemize}
 See \cite{mfmp2024hypergeometric} for further information. A more particular case that we will use in this paper is the polynomial on the left-hand side of the following factorization (\cite[equation~(1.7.1)]{MR2656096}, \cite[equation~(30)]{mfmp2024hypergeometric}):
 \begin{equation} \label{eq:zeroson1}
 \HGP{n}{b}{b+\frac{r}{n}} = (x-1)^{n-r} \HGP{r}{\frac{n}{r} (b-1)+1}{\frac{n}{r} b + 1}
 \end{equation}
 for $0 \leq r \leq n$. For $b > 1-\frac{r}{n^2}$, by the first bullet point above, the hypergeometric polynomial on the right-hand side of equation \eqref{eq:zeroson1} is in $\P_n([0,1])$, so the polynomial on the left-hand side is in $\P_n([0,1])$ as well.

 Similarly, one can pick parameters to produce Jacobi polynomials with roots at only $0$ and~$1$
 \[%\label{zeroson01}
 R_n^{(r)}(x) := \HGP{n}{r/n}{1}(x) =(x-1)^r x^{n-r} \text{ for } 0 \leq k \leq n.
 \]
 This is the characteristic polynomial of an orthogonal projection on $\cc^n$ with rank $r$. As such, it plays a special role in relation to free probability, where projections are an important type of noncommutative random variable.

 Some asymptotic results concerning Jacobi polynomials can be found in \cite{mfmp2024hypergeometric} and its references.
\end{Example}

\section{Even polynomials in finite free probability}
\label{sec:evenpol}

In this section, we will study in detail the basic properties of even polynomials, with special emphasis on their behavior under finite free convolutions.

\begin{Definition} \label{def:EvenPoly}
 We say that a polynomial $p \in \P_n$ is \emph{even} if one of the following equivalent statements holds:
 \begin{enumerate}\itemsep=0pt
 \item[(1)] for every root $\alpha$ of $p$, there is a root $\beta$ of $p$ with $\beta = -\alpha$ and $\mult(\beta) = \mult(\alpha)$;
 \item[(2)] $p$ is an even function if $n$ is even, or $p$ is an odd function if $n$ is odd;
 \item[(3)] $\sfe_k(p) = 0$ for all odd $0 \leq k \leq n$;
 \item[(4)] $p =\dil{-1} (p)$.
 \end{enumerate}
 We will denote by $\P_{2m}^E$ the set of all even polynomials of even degree $2m$.
\end{Definition}

The equivalence between (1), (2) and (3) above is well known. The equivalence of (3) and~(4) follows from noticing that $\dil{-1}$ simply changes the sign of the roots.

\begin{Remark}[even polynomials of odd degree]\label{rem:odd.degree}
To simplify the presentation, throughout this paper we will focus on studying even polynomials of even degree. However, the reader should keep in mind that the case of odd degree is completely analogous to the case of even degree, except that we need to add a root at $0$. Indeed, notice that if $p$ is an even polynomial of odd degree, then $p$ must have a root at $0$, implying that it is of the form $p(x)=xq(x)$ where $q$ is an even polynomial of even degree.
\end{Remark}

\begin{Remark}\label{rem:convolutions.even}
From Definition~\ref{def:FinFreeConv}, it is immediate that the set $\P_{2m}^E$ is closed under $\boxplus_{2m}$ and~$\boxtimes_{2m}$. Furthermore, $\P_{2m}^E$ is absorbing with respect to multiplicative convolution. These claims can be nicely summarized as follows
$\P_{2m}^E \boxplus_{2m} \P^E_{2m}\subset \P_{2m}^E$ and $
\P_{2m}^E \boxtimes_{2m} \P_{2m}\subset \P_{2m}^E$.
\end{Remark}

\subsection{Degree doubling operation}

A very simple way to construct even polynomials is by squaring the dummy variable. This is a~very natural operation, and has appeared in the context of finite free probability \cite{gribinski2022rectangular,MSS15}. Since we will extensively use this operation and its inverse, we will fix some notation.

\begin{Notation}[degree doubling operation]
 \label{nota:DegreeDoubling}
 Define $\lift_{m} \colon \P_m \to \P_{2m}^E$ by
$\lift_{m}(p) = p\bigl(x^2\bigr) $
 for $p \in \P_m$.
\end{Notation}

\begin{Notation}[even and odd parts] \label{nota:DegreeHalving}
 Define $\halfeven_{m} \colon \P_{2m} \to \P_m$ as follows: for $p \in \P_{2m}$, define~$\halfeven_{m}(p)$ by the coefficients
 \begin{equation} \label{eq:coeff.even.pol}
 \sfe_k(\halfeven_{m}(p)) = (-1)^k \sfe_{2k}(p).
 \end{equation}
\end{Notation}

Notice that the operations $\lift_{m}$ and $\halfeven_{m}$ are linear, and for $p \in \P_m$, the roots of $\lift_{m}(p)$ are
\[%\label{eq:roots.lift}
 -\sqrt{\lambda_1(p)},\sqrt{\lambda_1(p)}, -\sqrt{\lambda_{2}(p)},\sqrt{\lambda_{2}(p)}, \dots, -\sqrt{\lambda_{m}(p)}, \sqrt{\lambda_m(p)}.
\]
Another simple observation is that $\halfeven_{m} \circ \lift_{m} \colon \P_m \to \P_m$ is just the identity map, whereas the map $\lift_{m} \circ \halfeven_{m} \colon \P_{2m} \to \P_{2m}^E$ yields an even polynomial that has the same \emph{even} coefficients as the original polynomial. These observations provide a bijection that we will use constantly throughout this paper.

\begin{Fact} \label{fact:basic.Q}
 $\lift_{m}$ restricts to a bijection $\P_m(\rr_{\geq 0}) \to \P_{2m}^E(\rr)$, and the inverse of $\lift_{m}$ is $\halfeven_{m}$.
\end{Fact}

This is a finite free analogue of Lemma \ref{lem:Q.homeomorphism}. In Proposition~\ref{prop:Q.limit}, we will check that this bijection behaves well with respect to limits of empirical root distributions.

\subsection{Symmetrization}\label{ssec:symmetrization}

Another way to construct even polynomials is by taking the additive convolution of $p(x)$ with $p(-x)$, yielding an even polynomial of the same degree. This operation appeared naturally in~\cite{campbell2022commutators} when studying commutators in the context of finite free probability.

\begin{Notation}\label{not:sym}
 For $p \in \P_n$, the \emph{symmetrization} of $p$ is the polynomial
$%\label{symmdefinition}
 \mathrm{Sym}(p):=p\boxplus_n (\dil{-1} p ) $.
\end{Notation}

It follows directly from the definition that $\mathrm{Sym}(p) = \mathrm{Sym}(\dil{-1} p)$. Notice also that if $p \in \P_n^E$, then $\mathrm{Sym}(p) = p \boxplus_n p$. In the following lemma, we collect more properties of the symmetrization.

\begin{Lemma}\label{lem:basic.symmetrization}
 Let $p,q \in \P_n$ and $\alpha \in \cc$, and write $c_{\alpha}(x) := (x-\alpha)^n$. Then
 \begin{enumerate}\itemsep=0pt
 \item[$(1)$] $\mathrm{Sym}(p) \in \P_n^E$;
 \item[$(2)$] $\mathrm{Sym}(p \boxplus_n q) = \mathrm{Sym}(p) \boxplus_n \mathrm{Sym}(q)$;
 \item[$(3)$] $\mathrm{Sym}(\dil{\alpha} p) = \dil{\alpha} \mathrm{Sym}(p)$;
 \item[$(4)$] $\mathrm{Sym}(p \boxplus_n c_{\alpha}) = \mathrm{Sym}(p)$;
 \item[$(5)$] if $p \in \P_{2m}(\rr)$, then $\halfeven_{m}(\mathrm{Sym}(p)) \in \P_m(\rr_{\geq 0})$.
 \end{enumerate}
\end{Lemma}
\begin{proof}
 The proof of (1) follows from the observation that
 \[ \dil{-1} \mathrm{Sym}(p) = \dil{-1} ( p \boxplus_n \dil{-1} p ) = (\dil{-1} p ) \boxplus_n p = \mathrm{Sym}(p). \]
 Alternatively, one can use the formula for the coefficients of the additive convolution and notice that the negative signs of~$\dil{-1} p$ will generate cancellations, causing the odd coefficients of~$\mathrm{Sym}(p)$ to vanish.

 For (2), observe that both sides are equal to
$ p\boxplus_n \dil{-1} p \boxplus_n q \boxplus_n \dil{-1} q$.
 Part (3) is a direct consequence of the fact that dilation operation distributes over additive convolutions. For the proof of (4), first notice that
$\mathrm{Sym}(c_\alpha)=c_\alpha \boxplus_n c_{-\alpha}=c_0$.
 So by (2), we have $\mathrm{Sym}(p \boxplus_n c_{\alpha}) = \mathrm{Sym}(p)$. Finally, (5) follows from~(1) and Fact~\ref{fact:basic.Q}.
\end{proof}

\subsection{Multiplicative convolution}

The degree doubling operation behaves well with respect to multiplicative convolution.

\begin{Proposition} \label{prop.multiplicative.halving}
 For $p,q \in \P_{2m}$, we have
 \begin{equation} \label{eq:coeff.even.mult}
 \halfeven_{m}(p \boxtimes_{2m} q) = \halfeven_{m}(p) \boxtimes_m \halfeven_{m}(q) \boxtimes_m \HGP{m}{-\frac{1}{2m}}{1-\frac{1}{2m}}.
 \end{equation}
 Equivalently, we can express this in terms of the degree doubling operation: for $p,q\in \P_m$, we have
 \begin{equation} \label{eq:coeff.even.mult.lift}
 \lift_{m}(p \boxtimes_m q)= \lift_{m}(p) \boxtimes_{2m} \lift_{m}(q) \boxtimes_{2m} \lift_{m}\Bigl(\HGP{m}{1-\frac{1}{2m}, 1-\frac{1}{2m}}{-\frac{1}{2m}, -\frac{1}{2m}} \Bigr).
 \end{equation}
\end{Proposition}
\begin{proof}
 To prove equation \eqref{eq:coeff.even.mult}, we check the equality for every coefficient: for $0 \leq k \leq m$ the coefficient $\sfe_{k}(\halfeven_{m}(p \boxtimes_{2m} q))$ of the left-hand side polynomial is given by
 \begin{gather*}
 (-1)^k\sfe_{2k}(p \boxtimes_{2m} q) = \frac{(-1)^k}{\binom{2m}{2k}} \sfe_{2k}(p) \sfe_{2k}(q) \tag{Definition~\ref{def:FinFreeConv}} \\
 \qquad= (-1)^k(-1)^k \sfe_k(\halfeven_{m}(p)) (-1)^k \sfe_k(\halfeven_{m}(q)) \frac{(2k)!}{\falling{2m}{2k}} \tag{equation \eqref{eq:coeff.even.pol}} \\
 \qquad= \sfe_k(\halfeven_{m}(p)) \sfe_k(\halfeven_{m}(q)) \frac{(-1)^k k! \falling{k-\frac{1}{2}}{k}}{\falling{m}{k} \falling{m-\frac{1}{2}}{k}} \tag{equation \eqref{eq.falling.factorial.2m}}\\
\qquad = \frac{\sfe_k(\halfeven_{m}(p)) \sfe_k(\halfeven_{m}(q))}{\binom{m}{k} \binom{m}{k}} \binom{m}{k} \frac{ \falling{-\frac{1}{2}}{k}}{ \falling{m-\frac{1}{2}}{k}} \\
 \qquad= \sfe_k\Bigl(\halfeven_{m}(p) \boxtimes_m \halfeven_{m}(q)\boxtimes_m \HGP{m}{-\frac{1}{2m}}{1-\frac{1}{2m}}\Bigr). \tag{Definition~\ref{def:FinFreeConv}}
 \end{gather*}
 The proof of equation \eqref{eq:coeff.even.mult.lift} is analogous.
\end{proof}

\subsection{Additive and rectangular convolution}

The effect of taking even parts of finite free convolutions is somewhat more complicated in the additive case. Here, one finds that a variation of the rectangular convolution of Gribinski and Marcus \cite{gribinski2022rectangular} plays a key role. We learned of the generalized rectangular convolution from a talk by Gribinski \cite{TalkGribinski}. While working on this project, we became aware of another (equivalent) definition of generalized rectangular convolution given in \cite[Definition 1.11]{campbell2024free} in terms of differential operators.

\begin{Definition}[Generalized rectangular convolution]
 \label{def:rectangular.convolution}
 Let $m\in \nn$ and $\alpha \in \rr\setminus \{-1,\dots, -m\}$. For $p,q \in \P_m$, define the $(m,\alpha)$-rectangular convolution $p \boxplus_{m}^{\alpha} q \in \P_m$ as the polynomial determined by the coefficients
 \[ \sfe_k(p \boxplus_{m}^{\alpha} q)=\falling{m}{k} \falling{m+\alpha}{k} \sum_{i+j=k} \frac{\sfe_{i}(p)}{\falling{m}{i}\falling{m+\alpha}{i}} \frac{\sfe_{j}(q)}{\falling{m}{j}\falling{m+\alpha}{j}}. \]
\end{Definition}

\begin{Remark}
 It is worth mentioning that Definition~\ref{def:rectangular.convolution} is equivalent to the definitions presented in \cite{TalkGribinski} and \cite[Definition 1.11]{campbell2024free}. We prefer this presentation because one can readily rephrase it in terms of hypergeometric polynomials
 \[ \HGP{m}{\cdot}{1+\frac{\alpha}{m}} \boxtimes_m\left( p \boxplus_m^{\alpha} q \right)= \left( \HGP{m}{\cdot}{1+\frac{\alpha}{m}} \boxtimes_m p \right) \boxplus_m \left( \HGP{m}{\cdot}{1+\frac{\alpha}{m}} \boxtimes_m q \right). \]
\end{Remark}

With this definition in hand we can prove that the effect of taking even parts of finite free additive convolutions is related to the rectangular convolution in the case where $\alpha=-\frac{1}{2}$. This relation was mentioned without proof in \cite{TalkGribinski}; we include the proof here for the reader's convenience.

\begin{Proposition}
 \label{prop:rectangular.convolution}
 For $p,q \in \P_{2m}^E$, we have
 \begin{equation} \label{eq:coeff.even.add}
 \halfeven_{m}(p \boxplus_{2m} q)= \halfeven_{m}(p) \boxplus_{m}^{-1/2} \halfeven_{m}(q).
 \end{equation}
\end{Proposition}
\begin{proof}
 For $0 \leq k \leq m$, we have
 \begin{align*}
 \sfe_{k}^{(m)}(\halfeven_{m}(p \boxplus_{2m} q)) {}&= (-1)^k \sfe_{2k}^{(2m)}(p \boxplus_{2m} q) = (-1)^k\falling{2m}{2k}\sum_{i+j=k} \frac{\sfe_{2i}(p)}{\falling{2m}{2i}} \frac{\sfe_{2j}(q)}{\falling{2m}{2j}} \\
 &= (-1)^k 4^k \falling{m}{k} \falling{m-\frac{1}{2}}{k} \sum_{i+j=k} \frac{(-1)^i\sfe_{i}(\halfeven_{m}(p))}{4^i\falling{m}{i}\falling{m-\frac{1}{2}}{i}} \frac{(-1)^j\sfe_{j}(\halfeven_{m}(q))}{4^j\falling{m}{j}\falling{m-\frac{1}{2}}{j}} \\
 &= \sfe_k\bigl(\halfeven_{m}(p) \boxplus_{m}^{-1/2} \halfeven_{m}(q)\bigr).
 \end{align*}
 Since the coefficients match, the polynomials are the same.
\end{proof}

We should also mention that Gribinski \cite{TalkGribinski} conjectured that the generalized rectangular convolution preserves positive real roots.

\begin{Conjecture}
\label{conj.rectangular}
Given $\alpha>-1$, if $p,q \in \P_m(\rr_{\geq 0})$, then $p \boxplus_{m}^{\alpha} q \in \P_m(\rr_{\geq 0})$.
\end{Conjecture}
At the time of writing, the conjecture is only known to hold when $\alpha$ is a non-negative integer. The case $\alpha=0$ was proved in \cite[Theorem 3.1]{MSS15} and the cases $\alpha=1,2,3,\dots$ were proved in \cite[Theorem 2.3]{gribinski2022rectangular}.

To finish this section, we mention how to prove the case $\alpha=-\frac{1}{2}$ of Conjecture \ref{conj.rectangular} using Proposition~\ref{prop:rectangular.convolution}.

\begin{Corollary}
 If $p,q\in \P_m(\rr_{\geq 0})$, then $p\boxplus_{m}^{-1/2}q\in \P_m(\rr_{\geq 0})$.
\end{Corollary}

\begin{proof}
 Using \eqref{eq:coeff.even.add}, we can write
\smash{$ p \boxplus_{m}^{-1/2} q = \halfeven_{m}(\lift_{m}(p) \boxplus_{2m} \lift_{m}(q))$}.
 Since $p,q\in \P_m(\rr_{\geq 0})$, we have $\lift_{m}(p), \lift_{m}(q)\in \P_{2m}^E(\rr)$ and $\lift_{m}(p) \boxplus_{2m} \lift_{m}(q)\in \P_{2m}^E(\rr)$. Therefore, after applying the degree halving operation we conclude that \smash{$p \boxplus_{m}^{-1/2} q= \halfeven_{m}(\lift_{m}(p) \boxplus_{2m} \lift_{m}(q))\in \P_m(\rr_{\geq 0})$}.
\end{proof}

%%%%%%%%%%%%%%%%%%%%%%%%%%%%%%%%%%%%%%%%%%%%%%%%%%%%%
%%%%%%%%%%%%%%%%%%%%%%%%%%%%%%%%%%%%%%%%%%%%%%%%%%%%%
\section{Even hypergeometric polynomials}
\label{sec:EvenHyper}

The purpose of this section is to study the specifics of how even hypergeometric polynomials interact with the finite free convolution. Since this large class of polynomials contains several regions of parameters where the polynomials have all real, positive, or negative roots, understanding the multiplicative and additive convolutions in these cases will provide us with a large sample of even polynomials. For a summary of the results in this section using standard hypergeometric series notation ${}_iF_j$ see Section~\ref{sec:results.standard.notation}.

Our approach resembles that of \cite{mfmp2024hypergeometric}, with the difference that we want to study hypergeometric polynomials with the variable $x^2$ rather than $x$. This requires some adjustment to the convolution formulas, because of the dependence on the degree, which is now $2m$ instead of $m$.

\begin{Notation}[even hypergeometric polynomials]\label{not:even.hgp}
 Given $i,j,m\in \nn$, $\bm a =(a_1, \dots, a_i) \in \rr^i$ and~${\bm b =(b_1, \dots, b_j)\in \rr^j}$, write
\smash{$ \mathcal{H}_{m}^{E} \bigl[\begin{smallmatrix}\bm b\\ \bm a\end{smallmatrix}\bigr] := \lift_{m} \bigl( \mathcal{H}_{m}\bigl[\begin{smallmatrix}\bm b\\ \bm a\end{smallmatrix}\bigr] \bigr)$},
 where $\lift_{m}$ is the bijection from Notation~\ref{nota:DegreeHalving}. In terms of coefficients, we have
 \begin{align*}
 \sfe_{2k} \left( \HGPS{m}{\bm b}{\bm a} \right) = (-1)^k \sfe_k \left( \HGP{m}{\bm b}{\bm a} \right) = (-1)^k \binom{m}{k} \frac{\falling{\bm b m}{k}}{\falling{\bm a m}{k}}
= (-1)^k \binom{2m}{2k} \frac{\falling{k-\frac{1}{2}}{k} \falling{\bm b m}{k}}{\falling{m-\frac{1}{2}}{k}\falling{\bm a m}{k}}
 \end{align*}
 for $0 \leq k \leq m$.
\end{Notation}

\begin{Example}[Bernoulli polynomials]
 \label{exm:Bernoulli}
 The simplest even hypergeometric polynomial is
 \[ B_{2m}(x)=\HGPS{m}{\cdot}{\cdot}(x)=\bigl(x^2-1\bigr)^m. \]
 We will call this a Bernoulli polynomial because its empirical root distribution is a Bernoulli distribution with equal weights at $1$ and $-1$.
\end{Example}

\begin{Example}[Hermite polynomials] \label{exm:Hermite}
 Another important sequence of even hypergeometric polynomials is the one we encountered in Example \ref{exa:LaguerreHermite}:
 \[ H_{2m} = \dil{\sqrt{\frac{1}{m}}} \HGPS{m}{1-\frac{1}{2m}}{\cdot}. \]
\end{Example}

\subsection{Preliminary results}
To study the convolution of even hypergeometric polynomials, we first need to understand how the polynomials look in differential form.

A formula to write polynomials \smash{$\mathcal{H}_{m}\bigl[\begin{smallmatrix}\bm b\\ \bm a\end{smallmatrix}\bigr]$} as differential operators was implicitly found in \cite{mfmp2024hypergeometric}. The formula to write \smash{$\mathcal{H}_{m}^{E} \bigl[\begin{smallmatrix}\bm b\\ \bm a\end{smallmatrix}\bigr]$} in terms of a differential operator can be derived in a similar way, and it can be generalized for larger powers. We will first prove a general lemma for an arbitrary power and then we specialize to the cases we are concerned.

\begin{Lemma}
 \label{lem:differential.powers}
 Given a constant $c\in \rr$, integers $i,j,m,l\in \nn$, and tuples of parameters $\bm a =(a_1, \dots, a_i)\in \rr^i$ and $\bm b =(b_1, \dots, b_j)\in \rr^j$, if a polynomial has the following differential form
 \[ p(x)= \HGF{j}{i}{-m\bm b}{-m\bm a}{c\left(\frac{\partial}{\partial x}\right)^l} x^{lm} , \]
 then we can express it as the following hypergeometric polynomial
 \[ p(x) = \bigl( (-1)^{i+j+1} l^lc \bigr)^m \HGP{m}{\bm b, 1-\frac{1}{lm}, 1-\frac{2}{lm}, \dots, 1-\frac{l-1}{lm}}{\bm a} \left( \frac{(-1)^{i+j+1}}{l^lc} x^l \right). \]
\end{Lemma}
\begin{proof}
 We have
 \begin{align*}
 \HGF{j}{i}{-\bm bm}{-\bm am}{c\left(\frac{\partial}{\partial x}\right)^l} x^{lm} &= \sum_{k=0}^{\infty} \frac{\raising{-m\bm b}{k} }{\raising{-m\bm a}{k}} \frac{c^k}{k !}\left(\frac{\partial}{\partial x}\right)^{lk}x^{lm}\\
 &= \sum_{k=0}^m (-1)^{(i+j)k} \frac{\falling{m\bm b }{k} \falling{lm}{lk} }{\falling{m\bm a }{k} } \frac{c^k}{k !} x^{lm-lk}.
 \end{align*}
 If $s$ is not a multiple of $l$, then $\sfe_s(p) = 0$. Otherwise, we have
 \begin{align*}
 (-1)^{lk} \sfe_{lk}(p) &= \frac{\falling{m\bm b}{k}}{\falling{m\bm a}{k}} (-1)^{(i+j)k} \frac{c^k}{k!}\falling{lm}{lk} \\
 &= \frac{\falling{m\bm b}{k}}{\falling{m\bm a}{k}} \frac{c^k}{k!} (-1)^{(i+j)k} l^{lk} \falling{m}{k} \falling{m-\frac{1}{l}}{k} \cdots \falling{m-\frac{l-1}{l}}{k} \\
 &= \bigl((-1)^{i+j}l^lc\bigr)^m \binom{m}{k} \frac{\falling{m\bm b }{k} \falling{m-\frac{1}{l}}{k}\cdots \falling{m-\frac{l-1}{l}}{k} }{\falling{m\bm a }{k} } \left(\frac{ (-1)^{i+j}}{l^lc}\right)^{m-k}.
 \end{align*}
 So we have proven the claim
 \[ p(x) = \bigl((-1)^{i+j+1}l^lc\bigr)^m \HGP{m}{\bm b, 1-\frac{1}{lm}, 1-\frac{2}{lm}, \dots, 1-\frac{l-1}{lm}}{\bm a} \left( \frac{(-1)^{i+j+1}}{l^lc} x^l\right). \tag*{\qed} \]\renewcommand{\qed}{}
\end{proof}

As particular cases, we obtain the following.

\begin{Corollary} \label{cor:differential.form}
 With the assumptions of Lemma {\rm\ref{lem:differential.powers}}, in the case $l=1$, we have
 \[% \label{eq:differential.H}
 \HGF{j}{i}{-m\bm b}{-m\bm a}{c\frac{\partial}{\partial x}} x^m= \dil{c(-1)^{i+j+1}} \HGP{m}{\bm b}{\bm a}(x).
 \]
 In the case $l=2$, we have
 \begin{equation} \label{eq:differential.H.square}
 \HGF{j}{i}{-m\bm b}{-m\bm a}{c\left(\frac{\partial}{\partial x}\right)^2} x^{2m}= \dil{2\sqrt{c(-1)^{i+j+1}}} \HGPS{m}{\bm b, 1-\frac{1}{2m}}{\bm a}(x).
 \end{equation}
\end{Corollary}

It is worth emphasizing that in the right-hand side of \eqref{eq:differential.H.square}, we first double the degree of the hypergeometric polynomial and then dilate it. Notice that the same polynomial can be obtained by first dilating (by the square of the constant) and then doubling the degree.

With these formulas in hand, we can now readily generalize the last part of Theorem~\ref{thm:MFMP} and relate the product of hypergeometric series (evaluated in any power of $x$) to the additive convolution of hypergeometric polynomials (evaluated on the corresponding powers of $x$). The idea is to use the definition of $\boxplus_n$ in terms of differential operators and the fact that we just proved that hypergeometric series evaluated in differential operators applied to $x^n$ yield hypergeometric polynomials. We first provide the result in its more general form, and then specialize to the cases that we are more interested in.

\begin{Theorem}[additive convolution of hypergeometric polynomials] \label{thm:add.hyper.general}
 Let $c_1,c_2,c_3\in \rr$ be constants, and let $l_1,l_2,l_3,n\in \nn$ be numbers such that $l_k$ divides $n$ for $k=1,2,3$. Consider tuples $\bm a_1$, $\bm a_2$, $\bm a_3$, $\bm b_1$, $\bm b_2$, $\bm b_3 $ of sizes $i_1,i_2,i_3,j_1,j_2,j_3 \in \nn$, and assume that
 \[ \HGF{j_1}{i_1}{-n\bm b_1}{-n\bm a_1}{c_1 x^{l_1} } \HGF{j_2}{i_2}{-n\bm b_2}{-n\bm a_2}{c_2 x^{l_2}}=\HGF{j_3}{i_3}{-n\bm b_3}{-n\bm a_3}{c_3x^{l_3}}. \]
 Then, if for $k=1,2,3$ we consider the polynomials
 \[ p_k(x)= \bigl((-1)^{i_k+j_k+1}l_k^{l_k}c_k\bigr)^{\frac{n}{l_k}} \HGP{\frac{n}{l_k}}{l_k \bm b_k, 1-\frac{1}{n}, \dots, 1-\frac{l_k-1}{n}}{l_k \bm a_k} \left( \frac{(-1)^{i_k+j_k+1}}{l_k^{l_k}c_k} x^{l_k}\right), \]
 we get that $p_1 \boxplus_n p_2 =p_3$.
\end{Theorem}
\begin{proof}
 Fix $k \in \{ 1,2,3 \}$. From Lemma \ref{lem:differential.powers} applied to $c_k\in \rr$, integer values \smash{$i_k,j_k,\frac{n}{l_k},l_k\in \nn$}, and tuples of parameters $l_k \bm a_k \in \rr^i$ and $l_k \bm b_k \in \rr^j$, we know that a polynomial written in differential form
 \[ \HGF{j_k}{i_k}{-\frac{n}{l_k} l_k\bm b_k}{-\frac{n}{l_k} l_k\bm a_k}{c_k\left(\frac{\partial}{\partial x}\right)^{l_k}} x^{n} \]
 is precisely the polynomial
 \[ p_k(x)=\bigl((-1)^{i_k+j_k+1}l_k^{l_k}c_k\bigr)^{\frac{n}{l_k}} \HGP{\frac{n}{l_k}}{l_k\bm b_k, 1-\frac{1}{n}, \dots , 1-\frac{l_k-1}{n}}{l_k\bm a_k} \left( \frac{(-1)^{i_k+j_k+1}}{l_k^{l_k}c_k} x^{l_k}\right). \]
 Then the result follows from the definition of additive convolution using differential operators.\looseness=1
\end{proof}

%%%%%%%%%%%%%%%%%%%%%%%%%%%%%%%%%
\subsection{Symmetrization}
\label{ssec:hg.symm}

With the results from last section in hand, we are ready to study the symmetrization of some hypergeometric polynomials using some well-known results of products of hypergeometric functions.

\begin{Lemma}
 \label{lem:symmetrization.H}
 Consider tuples $\bm a$, $\bm a'$, $\bm b$, $\bm b'$ of sizes $i$, $i'$, $j$, $j'$, and assume that
 \begin{equation} \label{symcondition}
 \HGF{j}{i}{-2m\bm b}{-2m\bm a}{x} \HGF{j}{i}{-2m\bm b}{-2m\bm a}{ -x}=\HGF{j'}{i'}{-m \bm b'}{-m\bm a'}{c x^2}.
 \end{equation}
 Then we have
 \[
\mathrm{Sym}\left(\HGP{2m}{\bm b}{\bm a}\right)=\dil{2\sqrt{c(-1)^{i'+j'+1}}} \HGPS{m}{\bm b', 1-\frac{1}{2m}}{\bm a'}.
 \]
\end{Lemma}

\begin{proof}
The result follows from applying Theorem~\ref{thm:add.hyper.general} to the particular case where $n=2m$, $l_1=l_2=1$, $l_3=2$, $c_1=1$, $c_2=-1$, $c_3=c$, $\bm a_1=\bm a_2=\bm a$, and $\bm b_1=\bm b_2=\bm b$.
\end{proof}

Using this result, we can compute the symmetrization of certain hypergeometric polynomials using product identities for hypergeometric series of the form \eqref{symcondition}. Some of these formulas are elementary, like the product of binomial functions or the product of Bessel functions, while more involved ones can be found in works of Ramanujan, Preece, and Bailey. We use Grinshpan's survey \cite{grinshpan2013generalized} as a convenient reference. Specifically, in Proposition~\ref{prop:prod.hg.series.sym}, we reproduce equations~(9), (19), (21), (20), and (23) from \cite{grinshpan2013generalized} as equations \eqref{eq.Gri.9}--\eqref{eq.Gri.23}, respectively.

\begin{Proposition}[product of hypergeometric series]
 \label{prop:prod.hg.series.sym}
 Given real parameters $m$, $a$, $b$, $c$, $d$, the following identities hold:
 \begin{gather} %\label{eq.Gri.19} \label{eq.Gri.21} \label{eq.Gri.20}
 \HGF{1}{0}{-2mb}{\cdot}{x} \HGF{1}{0}{-2mb}{\cdot}{ -x}=\HGF{1}{0}{-2mb}{\cdot}{x^2} ;\label{eq.Gri.9}
\\
 \HGF{0}{1}{\cdot}{-2ma}{x} \HGF{0}{1}{\cdot}{-2ma}{-x}
 =\HGF{0}{3}{\cdot}{-2ma, -ma, -ma+\frac{1}{2}}{-\frac{x^2}{4}} ;
\\
 \HGF{1}{1}{-2mb}{-2ma}{x} \HGF{1}{1}{-2mb}{-2ma}{-x}
 =\HGF{2}{3}{-2mb, -2ma+2mb}{-2ma, -ma, -ma+\frac{1}{2}}{\frac{x^2}{4}} ;
\\
 \HGF{2}{0}{-2mb, -2md}{\cdot}{x} \HGF{2}{0}{-2mb, -2md}{\cdot}{-x}\nonumber
\\
\qquad
 =\HGF{4}{1}{-2mb, -2md, -m(b+d), -m(b+d)+\frac{1}{2}}{-2mb-2md}{4x^2} ;
\\
 \HGF{0}{2}{\cdot}{a, c}{x} \HGF{0}{2}{\cdot}{a, c}{-x}
 =\HGF{3}{8}{\frac{a+c-1}{3}, \frac{a+c}{3}, \frac{a+c+1}{3}}{a, c, \frac{a}{2}, \frac{c}{2}, \frac{a+1}{2}, \frac{c+1}{2}, \frac{a+c-1}{2}, \frac{a+c}{2}}{-\frac{27x^2}{64}}.\label{eq.Gri.23}
 \end{gather}
\end{Proposition}

Using Proposition~\ref{prop:prod.hg.series.sym} and Lemma \ref{lem:symmetrization.H}, we can compute the symmetrizations of various classical polynomials, such as Laguerre and Bessel polynomials, as well as their multiplicative convolutions. We collect these results in Table~\ref{tab:sym.hgp}.

\begin{table}[ht]\renewcommand{\arraystretch}{1.8}
 \centering
 \begin{tabular}{|c|c|c|}
 \hline
 Polynomial & $p$ & $\mathrm{Sym}(p)$ \\
 \hline
 Laguerre & $\HGP{2m}{b}{\cdot}$ & $\dil{2} \HGPS{m}{2b, 1-\frac{1}{2m}}{\cdot}$ \\
 \hline
 Bessel & $\HGP{2m}{\cdot}{a}$ & $\dil{i} \HGPS{m}{1-\frac{1}{2m}}{2a, a, a-\frac{1}{2m}}$ \\
 \hline
 Jacobi & $\HGP{2m}{b}{a}$ & $\HGPS{m}{1-\frac{1}{2m}, 2b, 2a-2b}{2a, a, a-\frac{1}{2m}}$ \\
 \hline
 \small Lag $\boxtimes_{2m}$ Lag & $\HGP{2m}{b, d}{\cdot}$ & $\dil{4} \HGPS{m}{1-\frac{1}{2m}, 2b, 2d, b+d, b+d-\frac{1}{2m}}{2b+2d}$ \\
 \hline
 \small Bes $\boxtimes_{2m}$ Bes & $\HGP{2m}{\cdot}{a, c}$ & \small $\dil{\frac{8i}{3\sqrt{3}}} \HGPS{m}{\frac{2}{3}\left(a+c+\frac{1}{2m}\right), \frac{2}{3}\left(a+c\right), \frac{2}{3}\left(a+c-\frac{1}{2m}\right), 1-\frac{1}{2m} }{2a, 2c, a, c, a-\frac{1}{2m}, c-\frac{1}{2m}, a+c+\frac{1}{2m}, a+c}$ \\
 \hline
 \end{tabular}

 \caption{Symmetrization of hypergeometric polynomials. Notice that here $i := \sqrt{-1}$ is a complex number. }
 \label{tab:sym.hgp}
\end{table}

We should emphasize that in the last row of Table~\ref{tab:sym.hgp}, the case corresponding to the multiplicative convolution of two Bessel polynomials, one must do the change of variable $a'=-2ma$ and $b'=-2bm$ in equation \eqref{eq.Gri.23} before applying Lemma \ref{lem:symmetrization.H}. Notice also, that in this case, if we let $a+c=\frac{-3}{2m}$, the formula simplifies to
\[
\mathrm{Sym}\left(\HGP{2m}{\cdot}{a, c}\right) = \dil{\frac{8i}{3\sqrt{3}}} \HGPS{m}{\frac{-4}{3m}, 1-\frac{1}{2m}}{2a, 2c, a, c, a-\frac{1}{2m}, c-\frac{1}{2m}} ,
\]
where $i := \sqrt{-1}$ is a complex number.

\subsection{Multiplicative convolution}

The multiplicative convolution of two even hypergeometric polynomials has a very nice expression that follows from Proposition~\ref{prop.multiplicative.halving} and Theorem~\ref{thm:MFMP}.

\begin{Proposition}[multiplicative convolution of even hypergeometric polynomials]
 \label{prop:EvenHyperMultConv}
 Consider tuples $\bm a_1$, $\bm a_2$, $\bm b_1$, $\bm b_2$ of sizes $i_1$, $i_2$, $j_1$, $j_2$. Then
 \[ \HGPS{m}{\bm b_1}{\bm a_1 }\boxtimes_{2m}\HGPS{m}{\bm b_2}{\bm a_2}= \HGPS{m}{-\frac{1}{2m}, \bm b_1 , \bm b_2 }{1-\frac{1}{2m}, \bm a_1, \bm a_2}. \]
\end{Proposition}

\begin{Example}[multiplicative convolution of two Bernoulli polynomials]
 \label{exm:product.Bernoulli}
 By Proposition~\ref{prop:EvenHyperMultConv}, the multiplicative convolution of two Bernoulli polynomials $B_{2m}$ from Example \ref{exm:Bernoulli} is given by
 \[ B_{2m}\boxtimes_{2m}B_{2m}= \HGPS{m}{-\frac{1}{2m}}{1-\frac{1}{2m}}. \]
 For $m>1$, this polynomial is not necessarily real-rooted; instead we can conclude that all the roots lie in the unit circle $\tt:=\{z\in \cc\mid |z|=1\}$. Indeed, since $B_{2m}$ has roots only at $1$ and~$-1$, then $B_{2m} \in \P_{2m}(\tt)$. Since $\P_{2m}(\tt)$ is closed under multiplicative convolution (see Lemma~\ref{lem:multiplicative.unit.circle}), we conclude
 \[ B_{2m} \boxtimes_{2m} B_{2m} = \HGPS{m}{-\frac{1}{2m}}{1-\frac{1}{2m}} \in \P_{2m}(\tt). \]
\end{Example}

\begin{Example}[multiplicative convolution of two Hermite polynomials] \label{exm:product.Hermite}
 The multiplicative convolution of two Hermite polynomials $H_{2m}$ from Example \ref{exm:Hermite} is given by
 \[ H_{2m}\boxtimes_{2m}H_{2m}= \dil{\frac{1}{m}} \HGPS{m}{-\frac{1}{2m}, 1-\frac{1}{2m}}{\cdot}. \]
 This polynomial is not real-rooted in general.
\end{Example}

\subsection{Additive convolution}

The additive convolution of even hypergeometric polynomials can now be described as a particular instance of Theorem~\ref{thm:add.hyper.general}, where we let all the powers to be squares.

\begin{Proposition}[additive convolution of even hypergeometric polynomials] \label{prop:EvenAdd}
 Consider tuples $\bm a_1$, $\bm a_2$, $\bm a_3$, $\bm b_1$, $\bm b_2$, $\bm b_3 $ of sizes $i_1$, $i_2$, $i_3$, $j_1$, $j_2$, $j_3$, and assume that
 \begin{equation} \label{eq:assumption.add.even.hyper}
 \HGF{j_1}{i_1}{-m\bm b_1}{-m\bm a_1}{c_1 x} \HGF{j_2}{i_2}{-m\bm b_2}{-m\bm a_2}{ c_2x}=\HGF{j_3}{i_3}{-m \bm b_3}{-m\bm a_3}{c_3x}.
 \end{equation}
 Then, if we let \smash{$s_k=\sqrt{c_k(-1)^{i_k+j_k+1}}$} for $k=1,2,3$, we have
 \begin{equation} \label{eq:add.even.hyper}
 \dil{s_1} \HGPS{m}{\bm b_1, 1-\frac{1}{2m}}{\bm a_1} \boxplus_{2m} \dil{s_2} \HGPS{m}{\bm b_2, 1-\frac{1}{2m}}{\bm a_2} =\dil{s_3} \HGPS{m}{\bm b_3, 1-\frac{1}{2m}}{\bm a_3}.
 \end{equation}
\end{Proposition}
\begin{proof}
 If we evaluate the assumption \eqref{eq:assumption.add.even.hyper} in $\frac{x^2}{4}$, we get
 \[ \HGF{j_1}{i_1}{-m\bm b_1}{-m\bm a_1}{c_1 \frac{x^2}{4}} \HGF{j_2}{i_2}{-m\bm b_2}{-m\bm a_2}{ c_2\frac{x^2}{4}}=\HGF{j_3}{i_3}{-m \bm b_3}{-m\bm a_3}{c_3\frac{x^2}{4}}. \]
 Then we can use Theorem~\ref{thm:add.hyper.general} with constants $\frac{c_1}{4},\frac{c_2}{4},\frac{c_3}{4}\in \rr$, integers $l_1=l_2=l_3=m$ and~${n=2m}$, and parameters $\frac{1}{2}\bm a_k$, $\frac{1}{2}\bm b_k$ for $k=1,2,3$. This gives $p_1\boxplus_{2m} p_2=p_3$, where
 \[ p_k=\dil{s_k} \HGPS{m}{\bm b_k, 1-\frac{1}{2m}}{\bm a_k}. \]
 Then equation \eqref{eq:add.even.hyper} follows from scaling the polynomials.
\end{proof}

\begin{Remark} Notice from Proposition~\ref{prop:rectangular.convolution} that the even parts of these polynomials can be related using the rectangular convolution $\boxplus_m^{-1/2}$.
\end{Remark}

Similar to Section~\ref{ssec:hg.symm}, we can compute the additive convolutions of certain hypergeometric polynomials using product identities of hypergeometric series that fit into the form of equation~\eqref{eq:assumption.add.even.hyper}. We again use Grinshpan's survey \cite{grinshpan2013generalized}. Specifically, in Proposition~\ref{prop:prod.hg.series.sum}, we reproduce equations (18), (7), (8), (10), and (11) from \cite{grinshpan2013generalized} as equations \eqref{eq.Gri.18}--\eqref{eq.Gri.11}, respectively.

\begin{Proposition}[product of hypergeometric series]\label{prop:prod.hg.series.sum}
 For real parameters $m$, $a$, $a_1$, $a_2$, $b_1$, $b_2$, $c_1$, and $c_2$, the following identities hold
 \begin{gather}%\label{eq.Gri.7} \label{eq.Gri.8}\label{eq.Gri.10}
 \HGF{0}{1}{\cdot}{-ma_1}{ x} \HGF{0}{1}{\cdot}{-ma_2}{x}
 = \HGF{2}{3}{\frac{-m(a_1+a_2)}{2},\frac{-m(a_1+a_2)-1}{2}}{-ma_1,-ma_2,-m(a_1+a_2)-1}{ 4x} ; \label{eq.Gri.18}
\\
 \HGF{0}{0}{\cdot}{\cdot}{-c_1 x} \HGF{0}{0}{\cdot}{\cdot}{-c_2 x}=\HGF{0}{0}{\cdot}{\cdot}{-(c_1 + c_2) x} ;
\\
 \HGF{1}{0}{-mb_1}{\cdot}{ x} \HGF{1}{0}{-mb_2}{\cdot}{x}=\HGF{1}{0}{-m(b_1+b_2)}{\cdot}{ x} ;
\\
 \HGF{1}{0}{-m(b_1+b_2-a)}{\cdot}{ x} \HGF{2}{1}{-m(a-b_1),-m(a-b_2)}{-ma}{x}\nonumber\\
 \qquad {} =\HGF{2}{1}{-mb_1,-mb_2}{-ma}{ x} ;
\\
 \left[\HGF{2}{1}{-mb_1,-mb_2}{-m(b_1+b_2)+1/2}{ x}\right]^2
 =\HGF{3}{2}{-2mb_1,-m(b_1+b_2),-2mb_2}{-m(b_1+b_2)+\frac{1}{2},-2m(b_1+b_2)}{ x}.\label{eq.Gri.11}
 \end{gather}
\end{Proposition}

Using Propositions \ref{prop:prod.hg.series.sum} and~\ref{prop:EvenAdd}, we can compute the additive convolutions of some even hypergeometric polynomials; we collect the results in Table~\ref{tab:sum.hgp}.

\begin{table}[ht]\renewcommand{\arraystretch}{1.8}
 \centering
 \begin{tabular}{|c|c|c|}
 \hline $p$ & $q$ & $p\boxplus_{2m} q$ \\
 \hline \hline
 $\HGPS{m}{1-\frac{1}{2m}}{a_1}$ & $\HGPS{m}{1-\frac{1}{2m}}{a_2}$ & \small $\dil{2} \HGPS{m}{\frac{a_1+a_2}{2}, \frac{a_1+a_2}{2}+\frac{1}{2m}, 1-\frac{1}{2m}}{a_1, a_2, a_1+a_2+\frac{1}{m}}$ \\
 \hline
 $\dil{\sqrt{c_1}} \HGPS{m}{1-\frac{1}{2m}}{\cdot}$ & $\dil{\sqrt{c_2}} \HGPS{m}{1-\frac{1}{2m}}{\cdot}$ & $\dil{\sqrt{c_1+c_2}} \HGPS{m}{1-\frac{1}{2m}}{\cdot}$ \\
 \hline
 $\HGPS{m}{b_1, 1-\frac{1}{2m}}{\cdot}$ & $\HGPS{m}{b_2, 1-\frac{1}{2m}}{\cdot}$ & $\HGPS{m}{b_1+b_2, 1-\frac{1}{2m}}{\cdot}$ \\
 \hline
 \small $\HGPS{m}{b_1+b_2-a, 1-\frac{1}{2m}}{\cdot}$ & \small $\HGPS{m}{a-b_1, a-b_2, 1-\frac{1}{2m}}{a}$ & $\HGPS{m}{b_1, b_2, 1-\frac{1}{2m}}{a}$ \\
 \hline
 $\HGPS{m}{b_1, b_2, 1-\frac{1}{2m}}{b_1+b_2-\frac{1}{2m}}$ & $\HGPS{m}{b_1, b_2, 1-\frac{1}{2m}}{b_1+b_2-\frac{1}{2m}}$ & $\HGPS{m}{2b_1, b_1+b_2, 2b_2, 1-\frac{1}{2m}}{b_1+b_2-\frac{1}{2m}, 2(b_1+b_2)}$ \\
 \hline
 \end{tabular}
 \caption{Sum of even hypergeometric polynomials.}
 \label{tab:sum.hgp}
\end{table}

\begin{Example}[additive convolution of Hermite polynomials]
 \label{exa:HermiteDouble}
 Notice that row~2 of Table~\ref{tab:sum.hgp} asserts that the additive convolution of two Hermite polynomials yields another Hermite polynomial, recovering a result which also follows from the machinery of \cite{arizmendi2018cumulants,marcus}.
\end{Example}

\begin{Example}[additive convolution of Bernoulli polynomials] \label{exa:BernoulliAddition}
 To compute the additive convolution of two Bernoulli polynomials, we can take $a_1=a_2=1-\frac{1}{2m}$ in row~1 of Table~\ref{tab:sum.hgp}. After a cancellation of the parameter $1-\frac{2}{m}$ appearing downstairs and upstairs in each polynomial, we obtain{\samepage
 \[ B_{2m}\boxplus_{2m}B_{2m}=\HGPS{m}{\cdot}{\cdot}\boxplus_{2m}\HGPS{m}{\cdot}{\cdot}= \dil{2} \HGPS{m}{1}{2}. \]
 Notice that the even part of the right-hand side is a Jacobi polynomial.}

 The special significance of this example, for us, is that it mirrors a basic example of free convolution which can be found in \cite[Example 12.8]{NS-book}. Namely, if $\mu = \frac{1}{2}(\delta_1 + \delta_{-1})$ is the measure with atoms at $\pm 1$ and mass $1/2$ each, the free convolution $\mu \boxplus \mu$ is a so-called \emph{arcsine distribution} (centered at $0$ and supported on $[-2,2]$). So the analogue of the arcsine distribution in finite free probability should be the dilation of the squared Jacobi polynomial \smash{$\dil{2} \mathcal{H}_{m}^{E}\bigl[\begin{smallmatrix}1\\ 2\end{smallmatrix}\bigr]$}.
\end{Example}

\section{Finite free commutators}
\label{sec:commutator}

One of the main insights of \cite{MSS15} is that the operations $\boxplus_n$ and $\boxtimes_n$, with their peculiar algebraic descriptions reviewed in Definition~\ref{def:FinFreeConv}, actually have very natural interpretations involving random matrices. Specifically, $p \boxplus_n q$ is the expected characteristic polynomial of $A + U B U^*$, where~${p(x) = \det(xI_n-A)}$ and $q(x) = \det(xI_n-B)$ are the characteristic polynomials of some diagonal $n \times n$ matrices $A$ and $B$, and $U$ is a random $n \times n$ unitary matrix. Similarly, $p \boxtimes_n q$ is the expected characteristic polynomial of $A U B U^*$.

A natural next step is to look at other polynomials in $A$ and $U B U^*$, and try to extract algebraic descriptions of their expected characteristic polynomials. In particular, knowing the historical development of free probability, one might gravitate towards the self-adjoint commutator $i(A U B U^* - U B U^* A)$. The algebraic description of this commutator operation, due to \cite{campbell2022commutators}, can be set up in purely algebraic terms involving $\boxplus_n$, $\boxtimes_n$, and a particular special polynomial.

\begin{Notation} \label{notation:Commutator}
 Let
 \[ z_n(x) := \sum_{k=0}^{\lfloor n/2 \rfloor} x^{n-2k} (-1)^k \binom{n}{2k} \falling{n}{k} \frac{k!}{(2k)!} \frac{n+1-k}{n+1} \]
 and for polynomials $p(x)$ and $q(x)$ with degree $n$, write
 \begin{equation} \label{eq:commutator}
 p \square_n q := \mathrm{Sym}(p)\boxtimes_n \mathrm{Sym}(q) \boxtimes_n z_n.
 \end{equation}
 The use of the symbol $\square$ is inspired by its use in \cite{NS98}.

 With $n=2m$, one can write
 \[ z_{2m} = \dil{\frac{1}{2}} \HGPS{m}{2, 2, 1-\frac{1}{2m}}{-\frac{1}{2m}, -\frac{1}{2m}, 2+\frac{1}{m}}. \]
 By Theorem~\ref{thm:realrootedness}, it follows that $z_{2m} \in \P_n^E(\rr)$.
\end{Notation}

The point of the operation $\square_n$ is that it encodes the expected characteristic polynomials of randomly rotated matrices.

\begin{Theorem}[\cite{campbell2022commutators}]
 \label{thm:Commutator}
 Let $A$ and $B$ be normal $n \times n$ matrices with characteristic polynomials~${p(x) =\chi(A):= \det[xI-A]}$ and $q(x) =\chi(B):= \det[xI-B]$. Then
 \[ \bE_U \chi [i(A U B U^* - U B U^* A)] = p(x) \square_n q(x), \]
 where $U$ is a random $n \times n$ unitary matrix.
\end{Theorem}

The realization of $z_n$ as a hypergeometric polynomial suggests a connection between special polynomials and analytic questions about finite free commutators. Our conjecture is that the commutator preserves real-rootedness in all cases.

\begin{Conjecture} \label{con:CommReal}
 For $p,q \in \P_n(\rr)$, we have $p \square_n q \in \P_n(\rr)$.
\end{Conjecture}

A general proof of this conjecture has turned out to be elusive, but we can provide some partial results. From now on, for the sake of simplicity we assume $n = 2m$, see Remark~\ref{rem:odd.degree}. First, we can rephrase the result of \cite{campbell2022commutators} in terms of our framework for even polynomials.

\begin{Proposition}[even part of commutator] \label{prop:commutator.even.part}
 For $p,q \in \P_{2m}(\rr)$, we have
 \[ \halfeven_{m}(p \square_{2m} q)= \halfeven_{m} (\mathrm{Sym}(p) ) \boxtimes_m \halfeven_{m} (\mathrm{Sym}(q) ) \boxtimes_m \dil{\frac{1}{4}} \HGP{m}{2, 2}{1-\frac{1}{2m}, 2+\frac{1}{m}}. \]
\end{Proposition}
\begin{proof}
 When one applies $\halfeven_{m}$ to the right-hand side of equation \eqref{eq:commutator} and uses Proposition~\ref{prop.multiplicative.halving}, one obtains the expression
 \begin{align*}
 \halfeven_{m}(\mathrm{Sym}(p)) \hspace{-0.95pt} \boxtimes_m\hspace{-0.95pt} \halfeven_{m}(\mathrm{Sym}(q)) \hspace{-0.95pt}\boxtimes_m \hspace{-0.95pt}\dil{\frac{1}{4}} \HGP{m}{2, 2, 1-\frac{1}{2m}}{-\frac{1}{2m}, -\frac{1}{2m}, 2+\frac{1}{m}}
\hspace{-0.95pt}\boxtimes_m \hspace{-0.95pt}\HGP{m}{-\frac{1}{2m}}{1-\frac{1}{2m}}\hspace{-0.95pt} \boxtimes_m \hspace{-0.95pt} \HGP{m}{-\frac{1}{2m}}{1-\frac{1}{2m}}.
 \end{align*}
 By Theorem~\ref{thm:MFMP}, the hypergeometric polynomials can be combined, and cancellation of parameters leaves
\smash{$ \mathcal{H}_{m}\bigl[ \begin{smallmatrix} 2, 2\\ 1-\frac{1}{2m}, 2+\frac{1}{m}\end{smallmatrix}\bigr] $},
 hence the claim.
\end{proof}

\begin{Remark}
 It is very important to notice that the polynomial
\smash{$ \mathcal{H}_{m}\bigl[ \begin{smallmatrix} 2, 2\\ 1-\frac{1}{2m}, 2+\frac{1}{m}\end{smallmatrix}\bigr]$},
 appearing in Proposition~\ref{prop:commutator.even.part}, does not belong to $\P_m(\rr_{\geq 0})$.

 Actually, if this polynomial were to be in $\P_m(\rr_{\geq 0})$, then Conjecture \ref{con:CommReal} would follow from part (3) of Theorem~\ref{thm:realrootedness}, after noticing that
$ \halfeven_{m}(\mathrm{Sym}(p) )\in \P_m(\rr_{\geq 0})$ and $\halfeven_{m}(\mathrm{Sym}(q) ) \in \P_m(\rr_{\geq 0})$
 by part (5) of Lemma \ref{lem:basic.symmetrization}.

 This means that to prove Conjecture \ref{con:CommReal}, one must find some kind of ``extra" real-rootedness in the symmetrizations $\halfeven_{m}(\mathrm{Sym}(p) )$. It is unclear if this extra positivity holds in general. With this idea in mind, however, we can explicitly formulate a theorem that is similar to Conjecture~\ref{con:CommReal}, but requires an extra assumption.
\end{Remark}

\begin{Theorem}
 \label{thm:partial-realrooted}
 Let $p,q \in \P_{2m}(\rr)$ and suppose that
 \[ \halfeven_{m}(\mathrm{Sym}(q)) = \HGP{m}{1-\frac{1}{2m}}{\cdot} \boxtimes_m r \]
 for some $r \in \P_m(\rr_{\geq 0})$. Then $p \square_{2m} q \in \P_{2m}(\rr)$.
\end{Theorem}
\begin{proof}
 Starting with the expression
 \[ \halfeven_{m}(p \square_{2m} q) = \halfeven_{m}(\mathrm{Sym}(p)) \boxtimes_m \halfeven_{m}(\mathrm{Sym}(q)) \boxtimes \dil{\frac{1}{4}} \HGP{m}{2, 2}{1-\frac{1}{2m}, 2+\frac{1}{m}} \]
 from Proposition~\ref{prop:commutator.even.part}, it suffices to show the polynomial on the right-hand side has all its roots in $\rr_{\geq 0}$.
 We can compute
 \begin{align*}
 \halfeven_{m}(p \square_{2m} q) &= \halfeven_{2m}(\mathrm{Sym}(p)) \boxtimes_m r \boxtimes_m \HGP{m}{1-\frac{1}{2m}}{\cdot} \boxtimes_m \dil{\frac{1}{4}} \HGP{m}{2, 2}{1-\frac{1}{2m}, 2+\frac{1}{m}} \\
 &= \halfeven_{m}(\mathrm{Sym}(p)) \boxtimes_m r \boxtimes_m \dil{\frac{1}{4}} \HGP{m}{2, 2}{2+\frac{1}{m}}
 \end{align*}
 and all three polynomials above are in $\P(\rr_{\geq 0})$. By Theorem~\ref{thm:realrootedness}, this shows $\halfeven_{2m}(p \square_{2m} q) \in \P_m(\rr_{\geq 0})$, and in turn that $p \square_{2m} q \in \P_{2m}(\rr)$.
\end{proof}

\begin{Remark}
 The assumption in Theorem~\ref{thm:partial-realrooted} is rather restrictive: it is satisfied by the Hermite, Laguerre, and Bessel polynomials that are of interest in finite free probability, but it fails for many other polynomials. However, it does all the work of guaranteeing real-rootedness of $p \square_{2m} q$ and allows one to put \emph{any} real-rooted polynomial in the other argument.
\end{Remark}

\subsection{Examples}

First, let us work out some examples that are in the scope of Theorem~\ref{thm:partial-realrooted}.

\begin{Example}[Hermite polynomials]
 \label{exa:HermiteCommutator}
 Let $H_{2m}$ be the Hermite polynomial defined in
 Example~\ref{exa:LaguerreHermite}. By Example~\ref{exa:HermiteDouble}, we have
 \[ \mathrm{Sym}(H_{2m}) = \dil{\sqrt{\frac{2}{m}}} \HGPS{m}{1-\frac{1}{2m}}{\cdot} \]
 so with $p = q = H_{2m}$, we have
 \begin{align*}
 \halfeven_{m}(H_{2m} \square_{2m} H_{2m}) &= \dil{\frac{2}{m}} \HGP{m}{1-\frac{1}{2m}}{\cdot} \boxtimes_m \dil{\frac{2}{m}} \HGP{m}{1-\frac{1}{2m}}{\cdot} \boxtimes_m \dil{\frac{1}{4}} \HGP{m}{2, 2}{1-\frac{1}{2m}, 2+\frac{1}{m}} \\
 &= \dil{\frac{1}{m^2}} \HGP{m}{1-\frac{1}{2m}, 2, 2}{2+\frac{1}{m}}.
 \end{align*}
 This polynomial is positive-rooted because
 \[ \HGP{m}{1-\frac{1}{2m}, 2, 2}{2+\frac{1}{m}} = \HGP{m}{1-\frac{1}{2m}}{2+\frac{1}{m}} \boxtimes_m \HGP{m}{2}{\cdot} \boxtimes_m \HGP{m}{2}{\cdot} ; \]
 the three polynomials on the right-hand side are well-known examples of positive-rooted Jacobi and Laguerre polynomials, reviewed in Examples~\ref{exa:Jacobi} and~\ref{exa:LaguerreHermite}, respectively. This shows that
 \[ H_{2m} \square_{2m} H_{2m} = \dil{\frac{1}{m}} \HGPS{m}{1-\frac{1}{2m}, 2, 2}{2+\frac{1}{m}} \in \P_{2m}(\rr). \]

 The limit in empirical root distribution of this finite free commutator will be described in Example \ref{exa:HermCommLimit}.
\end{Example}

\begin{Example}[Hermite and projection]
 \label{exa:HermProjCommutator}
 Now let us consider the commutator of a Hermite polynomial $H_{2m}$ with a projection-like polynomial \smash{$R_{2m}^{(m)}$}, whose roots are evenly split between $0$ and $1$. As computed above, we have
 \[ \mathrm{Sym}(H_{2m}) = \dil{\sqrt{2/m}} \HGPS{m}{1-\frac{1}{2m}}{\cdot} \qquad \text{and}\qquad \mathrm{Sym}\bigl(R_{2m}^{(m)}\bigr) = \HGPS{m}{1}{2} \]
 so
 \begin{align*}
 \halfeven_{m}\bigl(H_{2m} \square_{2m} R_{2m}^{(m)}\bigr) &= \dil{\frac{2}{m}} \HGP{m}{1-\frac{1}{2m}}{\cdot} \boxtimes_m \HGP{m}{1}{2} \boxtimes_m \dil{\frac{1}{4}} \HGP{m}{2, 2}{1-\frac{1}{2m}, 2+\frac{1}{m}} \\
 &= \dil{\frac{1}{2m}} \HGP{m}{1, 2}{2+\frac{1}{m}}.
 \end{align*}
 This polynomial is positive-rooted because
 \[ \HGP{m}{1, 2}{2+\frac{1}{m}} = \HGP{m}{1}{2+\frac{1}{m}} \boxtimes_m \HGP{m}{2}{\cdot} \]
 and the two polynomials on the right-hand side are again clear examples of positive-rooted Jacobi and Laguerre polynomials. This shows that
 \[ H_{2m} \square_{2m} R_{2m}^{(m)} = \dil{\frac{1}{\sqrt{2m}}} \HGPS{m}{1, 2}{2+\frac{1}{m}} \in \P_{2m}(\rr). \]
 The limit in empirical root distribution of this finite free commutator will be described in Example \ref{exa:HermProjCommLimit}.
\end{Example}

\begin{Example}[Laguerre polynomials]
 \label{exa:LaguerreCommutator}
 Recall the polynomial
\smash{$ L_n^{(\lambda)} = \dil{1/n} \mathcal{H}_{n}\bigl[\begin{smallmatrix} \lambda\\ \cdot\end{smallmatrix}\bigr] $},
 from Example \ref{exa:LaguerreHermite}. By row~1 of Table~\ref{tab:sym.hgp}, with $n=2m$, we have
 \[ \mathrm{Sym}\bigl(L_n^{(\lambda)}\bigr) = \dil{\frac{1}{m}} \HGPS{m}{2\lambda, 1-\frac{1}{2m}}{\cdot} \]
 so we have
 \begin{align*}
 \halfeven_{m}\bigl(L_n^{(\lambda)} \square_{2m} L_n^{(\mu)}\bigr) &= \dil{\frac{1}{m^2}} \HGP{m}{2\lambda, 1-\frac{1}{2m}}{\cdot} \boxtimes \dil{\frac{1}{m^2}} \HGP{m}{2\mu, 1-\frac{1}{2m}}{\cdot} \boxtimes_m \dil{\frac{1}{4}} \HGP{m}{2, 2}{1-\frac{1}{2m}, 2+\frac{1}{m}} \\
 &= \dil{\frac{1}{4m^4}} \HGP{m}{2\lambda, 2\mu, 2, 2, 1-\frac{1}{2m}}{2+\frac{1}{m}}.
 \end{align*}
 This polynomial is positive-rooted for $\lambda,\mu \geq \frac{1}{2}$ because
 \begin{align*}
 \HGP{m}{2\lambda, 2\mu, 2, 2, 1-\frac{1}{2m}}{2+\frac{1}{m}} &= \HGP{m}{2\lambda}{\cdot} \boxtimes_m \HGP{m}{2\mu}{\cdot} \boxtimes_m \HGP{m}{2}{\cdot} \boxtimes \HGP{m}{2}{\cdot}\boxtimes_m \HGP{m}{1-\frac{1}{2m}}{2+\frac{1}{m}} ;
 \end{align*}
 the first two are positive-rooted when $\lambda,\mu \geq \frac{1}{2}$, and the rest are positive-rooted as reviewed in Examples~\ref{exa:LaguerreHermite} and~\ref{exa:Jacobi}. This shows that
 \[ L_{2m}^{(\lambda)} \square_{2m} L_{2m}^{(\mu)} = \dil{\frac{1}{2m^2}} \HGPS{m}{2\lambda, 2\mu, 2, 2, 1-\frac{1}{2m}}{2+\frac{1}{m}} \in \P_{2m}(\rr). \]
 The limit in empirical root distribution of this finite free commutator will be described in Example \ref{exa:LaguCommLimit}.
\end{Example}

\begin{Example}[Bessel polynomials]
 \label{exa:BesselCommutator}
 Let $a,b < 0$ and let \smash{$p = C_n^{(a)}$} and \smash{$q = C_n^{(b)}$}, using the notation of Example \ref{exa:BesselRMP}. Then with $n=2m$, we have
 \[ \mathrm{Sym}\bigl(C_n^{(a)}\bigr) = \dil{2mi} \HGPS{m}{1-\frac{1}{2m}}{2a, a, a-\frac{1}{2m}} \]
 and
 \begin{align*}
 \halfeven_{m}\bigl(C_{2m}^{(a)} \square_{2m} C_{2m}^{(b)}\bigr) ={}& \dil{-4m^2} \HGP{m}{1-\frac{1}{2m}}{2a, a, a-\frac{1}{2m}} \boxtimes_m \dil{-4m^2} \HGP{m}{1-\frac{1}{2m}}{2b, b, b-\frac{1}{2m}} \\
 & \boxtimes_m \dil{\frac{1}{4}} \HGP{m}{2, 2}{1-\frac{1}{2m}, 2+\frac{1}{m}} \\
 ={}&\dil{4m^4} \HGP{m}{1-\frac{1}{2m}, 2, 2}{2a, a, a-\frac{1}{2m}, 2b, b, b-\frac{1}{2m}, 2+\frac{1}{m}}.
 \end{align*}
 This polynomial is positive-rooted because
 \begin{align*}
 \HGP{m}{1-\frac{1}{2m}, 2, 2}{2a, a, a-\frac{1}{2m}, 2b, b, b-\frac{1}{2m}, 2+\frac{1}{m}} ={}& \HGP{m}{2}{\cdot} \boxtimes_m \HGP{m}{2}{\cdot} \boxtimes_m \HGP{m}{1-\frac{1}{2m}}{2+\frac{1}{m}} \\
 & \boxtimes_m \HGP{m}{\cdot}{2a} \boxtimes_m \HGP{m}{\cdot}{a} \boxtimes_m \HGP{m}{\cdot}{a-\frac{1}{2m}} \\
 & \boxtimes_m \HGP{m}{\cdot}{2b} \boxtimes_m \HGP{m}{\cdot}{b} \boxtimes_m \HGP{m}{\cdot}{b-\frac{1}{2m}}.
 \end{align*}
 The Laguerre and Jacobi polynomials on the right-hand side are positive-rooted as explained in previous examples. The Bessel polynomials are all negative-rooted as explained in Example~\ref{exa:BesselRMP}, and the ``rule of signs" from Remark~\ref{rem:ruleofsigns} shows that the multiplicative convolution of six negative-rooted polynomials is positive-rooted. All told, this shows that
 \[ C_n^{(a)} \square_{2m} C_n^{(b)} = \dil{2m^2} \HGPS{m}{1-\frac{1}{2m}, 2, 2}{2a, a, a-\frac{1}{2m}, 2b, b, b-\frac{1}{2m}, 2+\frac{1}{m}} \in \P_{2m}(\rr). \]
\end{Example}

We can also work out examples that are \emph{not} covered by Theorem~\ref{thm:partial-realrooted}, in the sense that neither of the polynomials satisfy its factorization assumption. For these, we will need a particular result from the special function literature.

\begin{Remark} \label{rem:1112}
 The polynomial
 \smash{$\mathcal{H}_{m}\bigr[\begin{smallmatrix} 1, 1\\ 1-\frac{1}{2m}, 2+\frac{1}{m}\end{smallmatrix}\bigr]$}
 is positive-rooted. To see this, we can refer to~\cite[Theorem~3.6]{DJ02} for the fact that
 \[ \HGP{m}{1, 1}{1-\frac{1}{2m}, 2} \in \P_m((0,1)) \]
 and use Theorem~\ref{thm:MFMP} to write
 \[ \HGP{m}{1, 1}{1-\frac{1}{2m}, 2+\frac{1}{m}} = \HGP{m}{1, 1}{1-\frac{1}{2m}, 2} \boxtimes_m \HGP{m}{2}{2+\frac{1}{m}}. \]
 The latter polynomial is positive-rooted, as explained in \eqref{eq:zeroson1}, so the left-hand side is positive-rooted by Theorem~\ref{thm:realrootedness}.
\end{Remark}

\begin{Example}[projections and Bernoulli] \label{exa:ProjBernComm}
 Recall the special Jacobi polynomial $R_{2m}^{(r)}(x) = x^{2m-r} (x-1)^r$ from Example~\ref{exa:Jacobi}. To compute the finite free commutator of two such polynomials, let $0 \leq r,s \leq m$. By row~3 of Table~\ref{tab:sym.hgp}, with $n = 2m$, we have
 \[ \mathrm{Sym}\bigl(R_{2m}^{(r)}\bigr) = \HGPS{m}{\frac{r}{m}, 2-\frac{r}{m}}{2, 1} \qquad \text{and}\qquad \mathrm{Sym}\bigl(R_{2m}^{(s)}\bigr) = \HGPS{m}{\frac{s}{m}, 2-\frac{s}{m}}{2, 1} \]
 so by Proposition~\ref{prop:commutator.even.part}, we have
 \begin{align*}
 \halfeven_m\bigl(R_{2m}^{(r)} \square_{2m} R_{2m}^{(s)}\bigr)
 &= \halfeven_m\bigl(\mathrm{Sym}\bigl(R_{2m}^{(r)}\bigr)\bigr) \boxtimes_m \halfeven_m\bigl(\mathrm{Sym}\bigl(R_{2m}^{(s)}\bigr)\bigr) \boxtimes_m \dil{\frac{1}{4}} \HGP{m}{2, 2}{1-\frac{1}{2m}, 2+\frac{1}{m}} \\
 &= \HGP{m}{\frac{r}{m}, 2-\frac{r}{m}}{2, 1} \boxtimes_m \HGP{m}{\frac{s}{m}, 2-\frac{s}{m}}{2, 1} \boxtimes_m \dil{\frac{1}{4}} \HGP{m}{2, 2}{1-\frac{1}{2m}, 2+\frac{1}{m}} \\
 &= \dil{\frac{1}{4}} \HGP{m}{\frac{r}{m}, 2-\frac{r}{m}, \frac{s}{m}, 2-\frac{s}{m}}{1, 1, 1-\frac{1}{2m}, 2+\frac{1}{m}}
 \end{align*}
 and
 \[ R_{2m}^{(r)} \square_{2m} R_{2m}^{(s)} = \dil{\frac{1}{2}} \HGPS{m}{\frac{r}{m}, 2-\frac{r}{m}, \frac{s}{m}, 2-\frac{s}{m}}{1, 1, 1-\frac{1}{2m}, 2+\frac{1}{m}}. \]
 With $r=s=m$, i.e., with the roots evenly split between $0$ and $1$, the above reads as
 \[ R_{2m}^{(m)} \square_{2m} R_{2m}^{(m)} = \dil{\frac{1}{2}} \HGPS{m}{1, 1}{1-\frac{1}{2m}, 2+\frac{1}{m}}. \]

 Notice that the polynomial $R_{2m}^{(m)}(x)=x^{m} (x-1)^m$ is very similar to the Bernoulli polynomial~${B_{2m}(x) = (x-1)^m (x+1)^m}$, introduced in Example \ref{exm:Bernoulli}. Indeed, we can get one from the other by performing a shift and dilation:
% \label{eq:BRshiftdilation}
\smash{$ B_{2m} = \bigl(\dil{2} R_{2m}^{(m)}\bigr) \boxplus_{2m} \delta_1 $},
 where $\delta_1(x)=(x+1)^{2m}$. This means that symmetrizations of these two polynomials are the same up to a dilation by 2, and the commutators coincide up to a dilation by 4. Specifically, by Lemma \ref{lem:basic.symmetrization}, one can check that
 \begin{align*}
\mathrm{Sym}(B_{2m}) &= \mathrm{Sym}\bigl(\bigl(\dil{2} R_{2m}^{(m)}\bigr) \boxplus_{2m} \delta_1\bigr) = \mathrm{Sym}\bigl(\dil{2} R_{2m}^{(m)}\bigr) = \dil{2}\mathrm{Sym}\bigl(R_{2m}^{(m)}\bigr).
 \end{align*}
 Thus, by equation \eqref{eq:commutator} we have
 \begin{align*}
 B_{2m} \square_{2m} B_{2m} &= \mathrm{Sym}(B_{2m}) \boxtimes_{2m} \mathrm{Sym}(B_{2m})\boxtimes_{2m} z_{2m} \\
 &= \dil{2} \mathrm{Sym}\bigl(R_{2m}^{(m)}\bigr) \boxtimes_{2m} \dil{2}\mathrm{Sym}\bigl(R_{2m}^{(m)}\bigr)\boxtimes_{2m} z_{2m}= \dil{4} \bigl( R_{2m}^{(m)} \square_{2m} R_{2m}^{(m)} \bigr),
 \end{align*}
 so we conclude that
 \[ B_{2m} \square_{2m} B_{2m} = \dil{2} \HGPS{m}{1, 1}{1-\frac{1}{2m}, 2+\frac{1}{m}}. \]
 As explained in Remark~\ref{rem:1112}, the polynomial appearing as the finite free commutator in these examples is real-rooted. The limit empirical root distributions of these finite free commutators will be described in Examples~\ref{exa:BernCommLimit} and~\ref{exa:ProjCommLimit}.
\end{Example}

\begin{Example}[Jacobi polynomials]
 \label{exa:JacobiCommutator}
 For this example, we omit details, since they are cumbersome and identical to previous examples. If \smash{$p = \HGP{2m}{b}{a}$} and \smash{$q = \HGP{2m}{d}{c}$}, then
 \[ p \square_{2m} q = \dil{1/2} \HGPS{m}{2b, 2a-2b, 2d, 2c-2d, 2, 2, 1-\frac{1}{2m}}{2a, a, a-\frac{1}{2m}, 2c, c, c-\frac{1}{2m}, 2+\frac{1}{m}}. \]
 For example, with $b=d=1$ and $a=c=2$, the above is
 \[ \dil{1/2} \HGPS{m}{2, 2, 2, 2, 1-\frac{1}{2m}}{4, 4, 2-\frac{1}{2m}, 2-\frac{1}{2m}, 2+\frac{1}{m}}. \]
\end{Example}

We collect these examples in Table~\ref{tab:com.hgp}.

\begin{table}[ht]\renewcommand{\arraystretch}{1.8}
 \centering
 \begin{tabular}{|c|c|c|c|}
 \hline $p$ & $q$ & $p \square_{2m} q$ \\
 \hline \hline
 $\HGP{2m}{\frac{1}{2}}{1}$ & $\HGP{2m}{\frac{1}{2}}{1}$ & $\dil{\frac{1}{2}} \HGPS{m}{1, 1}{1-\frac{1}{2m}, 2+\frac{1}{m}}$ \\
 \hline $\HGPS{m}{1-\frac{1}{2m}}{\cdot}$ & $ \HGPS{m}{1-\frac{1}{2m}}{\cdot}$ & $\dil{\frac{1}{2}} \HGPS{m}{1-\frac{1}{2m}, 2, 2}{2+\frac{1}{m}}$ \\
 \hline
 $\HGPS{m}{1-\frac{1}{2m}}{\cdot}$ & $\HGP{2m}{\frac{1}{2}}{1}$ & $\dil{\frac{1}{2}} \HGPS{m}{1, 2}{2+\frac{1}{m}}$ \\
 \hline
 $\HGP{2m}{b}{\cdot} $ & $\HGP{2m}{d}{\cdot} $ & $\dil{\frac{1}{2}} \HGPS{m}{2b, 2d, 2, 2, 1-\frac{1}{2m}}{2+\frac{1}{m}}$ \\
 \hline
 $\HGP{2m}{1}{2}$ & $\HGP{2m}{1}{2}$ & $\dil{\frac{1}{2}} \HGPS{m}{2, 2, 2, 2, 1-\frac{1}{2m}}{4, 4, 2-\frac{1}{2m}, 2-\frac{1}{2m}, 2+\frac{1}{m}}$ \\
 \hline
 \end{tabular}
 \caption{Examples of finite free commutators.} \label{tab:com.hgp}
\end{table}

\section{Asymptotics and connection to free probability}\label{sec:asymptotics}

Most of the results from previous sections have counterparts in free probability, if we let the degree $n$ tend to infinity while the empirical root distribution tends to a compact measure. Some of these results are well known in free probability, and in fact served as original motivation for the results in finite free probability. However, to the best of our knowledge, many examples are new, and might shed some light on future directions in this area.

We begin by explicitly proving the intuitive fact that when we let $m\to\infty$, the map $\lift_{m}$ from Notation~\ref{nota:DegreeDoubling} tends to the map $S$ from Notation~\ref{nota:Q.measures}.
Thus, the bijection between $\P_m(\rr_{\geq 0})$ and~$\P_{2m}^E(\rr)$ from Fact \ref{fact:basic.Q}, in the limit becomes the bijection between $\MM(\rr_{\geq 0})$ and $\MM^E(\rr)$ from Lemma \ref{lem:Q.homeomorphism}.

\begin{Proposition}[even polynomials approximate symmetric measures] \label{prop:Q.limit}
 Let $\mfp = (p_m)_{m \geq 1}$ be a~converging sequence of positive-rooted polynomials. Then $\sqrt{\mfp} := (\lift_{m}(p_{2m}))_{m \geq 1}$ is a converging sequence of even real-rooted polynomials, with $\rho(\sqrt{\mfp}) = S(\rho(\mfp))$.

 Similarly, if $\mfq=(q_{2m})_{m\geq 1}$ is a converging sequence of even real-rooted polynomials, then $\mfq^2:=(\halfeven_{m}(q_{2m}))_{m\geq 1}$ is a converging sequence of positive-rooted polynomials with $\rho\bigl(\mfq^2\bigr) = Q(\rho(\mfq))$.
\end{Proposition}

\begin{proof}
 The basic tool involved in the proof of this proposition is the fact \cite[Theorem 2.7]{Billingsley-convergence} that pushforwards along continuous functions preserve weak limits.

 For the first claim, the assumptions amount to $\lim_{m \to \infty} \rho(p_m) = \rho(\mfp)$. Then
 \[ \lim_{m \to \infty} S_+(\rho(p_m)) = S_+(\rho(\mfp)) \qquad \text{and}\qquad \lim_{m \to \infty} S_-(\rho(p_m)) = S_-(\rho(\mfp)) \]
 since $\rr_{\geq 0} \to \rr_{\geq 0} \colon t \mapsto \sqrt{t}$ and $\rr_{\geq 0} \to \rr_{\leq 0} \colon t \mapsto -\sqrt{t}$ are continuous. This shows that
 \[ \lim_{m \to \infty} S(\rho(p_m)) = S(\rho(\mfp)). \]
 Since $S(\rho(p_m)) = \rho(\lift_{m}(p_m))$, we have
 \[ \lim_{m \to \infty} \rho(\lift_{m}(p_m)) = S(\rho(\mfp)) , \]
 hence $\sqrt{\mfp} = (\lift_{m}(p_m)_{m \geq 1}$ is a converging sequence with $\rho(\sqrt{\mfp}) = S(\rho(\mfp))$.

 For the second claim, suppose that $\mfq = (q_{2m})_{m \geq 1}$ is a converging sequence of even real-rooted polynomials, i.e., $\lim_{m \to \infty} \rho(q_{2m}) = \rho(\mfq)$. Then since $\rr \to \rr_{\geq 0} \colon t \mapsto t^2$ is continuous, we have
 \[ \lim_{m \to \infty} Q(\rho(q_{2m})) = Q \bigl( \lim_{m \to \infty} \rho(q_{2m}) \bigr) = Q(\rho(\mfq)). \]
 Since $Q(\rho(q_{2m}) = \rho(\halfeven_{m}(q_{2m}))$, this shows that
 \[ \lim_{m \to \infty} \rho(\halfeven_{m}(q_{2m})) = Q(\rho(\mfq)) , \]
 i.e., $\mfq^2 = (\halfeven_{m}(q_{2m}))_{m \geq 1}$ is a converging sequence with $\rho\bigl(\mfq^2\bigr) = Q(\rho(\mfq))$.
\end{proof}

Using this result, we can study the limiting root distributions of even polynomial. Specifically, the examples from Section~\ref{sec:EvenHyper} involving hypergeometric polynomials can be studied asymptotically, using free convolutions and even measures. Recall from Theorem~\ref{thm:HypLimit} and \mbox{Notation}~\ref{nota:rhoparamseven} that the asymptotic root distribution of hypergeometric polynomials is given by a~$S$-rational measure \smash{$\rho \bigl[\begin{smallmatrix} b_1,\dots, b_j\\ a_1,\dots, a_i\end{smallmatrix} \bigr]$}, determined by having an $S$-transform of the form \smash{$\frac{(z+a_1)\cdots (z+a_i) }{(z+b_1)\cdots(z+b_j)}$}. Recall also that we denote the square root of an $S$-rational function by
\smash{$\rho^{E} \bigl[\!\begin{smallmatrix} b_1,\dots, b_j\\ a_1,\dots, a_i\end{smallmatrix}\! \bigr]$}.

We now state an analogue of Theorem~\ref{thm:HypLimit} that holds for even hypergeometric polynomials.

\begin{Theorem}
 \label{thm:EvenHypLimit}
 For non-negative integers $i$, $j$ consider tuples of parameters $\bm A=(A_1,\dots, A_i) \in (\rr\setminus [0,1))^i$ and $\bm B=(B_1,\dots, B_j)\in (\rr\setminus \{0\})^j$. Assume that $\mfp=(p_m)_{m\geq 0}$ is a sequence of polynomials given by
 \[%\label{sequencePneven}
 p_m=\dil{\sqrt{m^{i-j}}}\HGPS{m}{\bm b_m}{\bm a_m} \in \PP^E_{2m}(\rr),
 \]
 where the tuples of parameters $\bm a_m \in \rr^i$ and $\bm b_m\in \rr^j$ have a limit given by
% \label{assumptionlimitseven}
 $\lim_{m\to\infty} \bm a_m=\bm A$, and $\lim_{m\to\infty} \bm b_m=\bm B$.
 Then $\mfp$ is a converging sequence $($in the sense of Definition~{\rm\ref{def:converging})} that converges to the measure
\smash{$\rho^{E} \bigl[\begin{smallmatrix} B_1,\dots, B_j\\ A_1,\dots, A_i\end{smallmatrix}\bigr]\in \MM^E(\rr)$}.
\end{Theorem}

\begin{proof}
Follows from Proposition~\ref{prop:Q.limit} and Theorem~\ref{thm:HypLimit}.
\end{proof}

\begin{Example}[Hermite polynomials and Laguerre polynomials]
 \label{exa:HermiteSquared}
 In free probability, it is known \cite[Proposition 12.13]{NS-book} that the square of a semicircular element has a free Poisson distribution with rate $1$. We can observe a similar phenomenon by applying our operations $\halfeven_m$ and $\lift_m$ to Hermite and Laguerre polynomials; namely, it corresponds to the observation in Example \ref{exa:LaguerreHermite} that Hermite polynomials are related to certain Laguerre polynomials with squared variables.

 More precisely, recall that the respective finite analogues of the free Poisson distribution with rate $1$ and the semicircular distribution with rate $2$ are the polynomials
 \[ L_n^{(1)}(x) = n^{-n} \HGP{n}{1}{\cdot}(nx) \qquad \text{and} \qquad H_{2m}(x) = m^{-m} \HGP{m}{1-\frac{1}{2m}}{\cdot}\bigl(mx^2\bigr). \]
 We have
 \[ \halfeven_{m}(H_{2m}) = m^{-m} \HGP{m}{1-\frac{1}{2m}}{\cdot}(mx) \]
 and since $\lim_{m \to \infty} (1-\frac{1}{2m}) = 1$, by Theorem~\ref{thm:HypLimit}, we have
 \[ \lim_{m \to \infty} \rho(\halfeven_{m}(H_{2m})) = \SRM{1}{\cdot} = \lim_{n \to \infty} \rho\bigl(L_n^{(1)}\bigr). \]
 We thus recover the aforementioned result concerning the square of a semicircular variable.
\end{Example}

\subsection{Examples of free symmetrization}

The symmetrization operation from Notation~\ref{not:sym} tends to the corresponding symmetrization operation in free probability, defined as follows:

\begin{Notation}
 \label{not:free.sym}
 For $\mu \in \MM(\rr)$, define the free symmetrization of $\mu$ by
 \[ \mathrm{Sym}(\mu) := \mu \boxplus (\dil{-1} \mu) \in \MM^E(\rr). \]
\end{Notation}

\begin{Remark}
 \label{rem:free.sym}
 Despite being a simple operation, to the best of our knowledge $\mathrm{Sym}(\mu)$ has not been studied systematically in free probability. Some particular instances where this operation has appeared include \cite[Final remarks]{hinz2010multiplicative}, \cite[p.~274]{mlotkowski2011symmetrization}, \cite[Example 25]{perez2012free}, and \cite[Example 7.9\,(3)]{AHS13}.

 Notice also that the term of \emph{free symmetrization} has appeared before in the literature, but denoting a different notion, like in \cite[Section 4]{hinz2010multiplicative}. It is unclear if there is a connection between both notions.
\end{Remark}

It would be interesting to study $\mathrm{Sym}(\mu)$ in detail, due to its potential connections to the commutator. Below we present some basic properties, that are analogous to those of the symmetrization of polynomials.

\begin{Lemma}
 \label{lem:symmetrization.meas}
 Let $\mu,\nu \in \MM(\rr)$ and $\alpha \in \rr$, and let $\delta_{\alpha}$ be the Dirac measure with and atom of mass one in $\alpha$. Then
 \begin{enumerate}\itemsep=0pt
 \item[$(1)$] $\mathrm{Sym}(\mu) \in \MM^E$;
 \item[$(2)$] $\mathrm{Sym}(\mu \boxplus \nu) = \mathrm{Sym}(\mu) \boxplus \mathrm{Sym}(\nu)$;
 \item[$(3)$] $\mathrm{Sym}(\dil{\alpha} \mu) = \dil{\alpha} \mathrm{Sym}(\mu)$;
 \item[$(4)$] $\mathrm{Sym}(\mu \boxplus \delta_{\alpha}) = \mathrm{Sym}(\mu)$;
 \item[$(5)$] $\halfeven(\mathrm{Sym}(\mu)) \in \MM(\rr_{\geq 0})$.
 \end{enumerate}
\end{Lemma}
\begin{proof}
 The proof is analogous to that of Lemma \ref{lem:basic.symmetrization}; one just needs to substitute the polynomials for measures, and change $\boxplus_n$ to $\boxplus$.
\end{proof}

Alternatively one could prove Lemma \ref{lem:symmetrization.meas} by letting $n\to\infty$ in Lemma \ref{lem:basic.symmetrization}, and noticing that the symmetrization of polynomials tends to the free symmetrization. Since this last fact will we useful later, let us state it explicitly.

\begin{Lemma}
 Let $\mfp=(p_m)_{m\geq 1} \subset \P(\rr_{\geq 0})$ be a converging sequence of polynomials. Then $\mathrm{Sym} (\mfp):=(\mathrm{Sym}(p_m))_{m\geq 1} \subset \P^E(\rr)$ is also a converging sequence of polynomials, and $\rho(\operatorname{Sym}(\mfp))\allowbreak = \operatorname{Sym}(\rho(\mfp))$.
\end{Lemma}
\begin{proof}
The claim follows from Theorem~\ref{thm:finiteAsymptotics} and the fact that $(\dil{-1} p_m)_{m\geq 1}$ is a converging sequence of polynomials with limit $\dil{-1}\rho(\mfp)$.
\end{proof}

As a direct application of this lemma we can take limits in Table~\ref{tab:sym.hgp} to obtain their analogues in free probability. We present these results in Table~\ref{tab:sym.meas}.

\begin{table}[ht]\renewcommand{\arraystretch}{1.8}
 \centering
 \begin{tabular}{|c|c|c|}
 \hline Measure & $\mu$ & $\mathrm{Sym}(\mu)$ \\
 \hline \hline
 MP & $\SRM{b}{\cdot}$ & $\dil{2} \SRMS{2b, 1}{\cdot}$ \\
 \hline
 RMP & $\SRM{\cdot}{a}$ & $\dil{i} \SRMS{1}{2a, a, a}$ \\
 \hline
 Free beta & $\SRM{b}{a}$ & $\SRMS{1, 2b, 2a-2b}{2a, a, a}$ \\
 \hline MP $\boxtimes$ MP & $\SRM{b, d}{\cdot}$ & $\dil{4} \SRMS{1, 2b, 2d, b+ d, b+ d}{2b+2d}$ \\
 \hline
 RMP $\boxtimes$ RMP & $\SRM{\cdot}{a, c}$ & $\dil{\frac{8i}{3\sqrt{3}}} \SRMS{\frac{2}{3}\left(a+c\right), \frac{2}{3}\left(a+c\right), \frac{2}{3}\left(a+c\right), 1}{2a, 2c, a, c, a, c, a+c, a+c} $ \\
 \hline
 \end{tabular}
 \caption{Symmetrization of $S$-rational measures. Here MP stands for Marchenko--Pastur, while RMP stands for reversed Marchenko--Pastur distribution.} \label{tab:sym.meas}
\end{table}

Similarly, in Table~\ref{tab:sum.meas} we compute the free additive convolution of some symmetric measures. These results follow from letting $m\to\infty$ in Table~\ref{tab:sum.hgp}.

\begin{table}[ht]\renewcommand{\arraystretch}{1.8}
 \centering
 \begin{tabular}{|c|c|c|}
 \hline $\mu$ & $\nu$ & $\mu\boxplus \nu$ \\
 \hline \hline
 $\SRMS{1}{a_1}$ & $\SRMS{1}{a_2}$ & $\dil{2} \SRMS{\frac{a_1+a_2}{2}, \frac{a_1+a_2}{2}, 1}{a_1, a_2, a_1+a_2}$ \\
 \hline $\dil{\sqrt{c_1}} \SRMS{1}{\cdot}$ & $\dil{\sqrt{c_2}} \SRMS{1}{\cdot}$ & $\dil{\sqrt{c_1+c_2}} \SRMS{1}{\cdot}$ \\
 \hline
 $\SRMS{b_1, 1}{\cdot}$ & $\SRMS{b_2, 1}{\cdot}$ & $\SRMS{b_1+b_2, 1}{\cdot}$ \\
 \hline
 $\SRMS{b_1+b_2-a, 1}{\cdot}$ & $\SRMS{a-b_1, a-b_2, 1}{a}$ & $\SRMS{b_1, b_2, 1}{a}$ \\
 \hline
 $\SRMS{b_1, b_2, 1}{b_1+b_2}$ & $\SRMS{b_1, b_2, 1}{b_1+b_2}$ & $\SRMS{2b_1, 2b_2, 1}{2(b_1+b_2)}$ \\
 \hline
 \end{tabular}
 \caption{Free additive convolutions of even measures.} \label{tab:sum.meas}
\end{table}

Notice that row~2 in Table~\ref{tab:sum.meas} is just the well-known fact that the free convolution of semicircular measures yields another semicircular measure.

Also, recall that if $\mu\in\MM^E(\rr)$ than $\mathrm{Sym}(\mu)=\mu\boxplus \mu$. Thus, when we let $\mu=\nu$ in Table~\ref{tab:sum.meas} we obtain results on the symmetrization of even measures.

A particular case, that will be useful later, is when we let $a_1=a_2=1$ in row~1 of Table~\ref{tab:sum.hgp}. This corresponds to studying the limit of Example \ref{exa:BernoulliAddition}, and recovers the well-known fact that the additive convolution of Bernoulli distributions gives an arcsine distribution.

\subsection{Examples of commutators}

\begin{Example}[continuation of Example \ref{exa:HermiteCommutator}]
 \label{exa:HermCommLimit}
 In Example \ref{exa:HermiteCommutator}, we found that
 \[ H_{2m} \square_{2m} H_{2m} = \dil{\frac{1}{m}} \HGPS{m}{2, 2, 1-\frac{1}{2m}}{2+\frac{1}{m}} \in \P_{2m}(\rr). \]
 In \cite{NS98}, it is shown that the commutator of freely independent semicircular elements has the same distribution as the difference of two freely independent free Poisson elements. Our computation matches this in the limit: from Table~\ref{tab:sym.hgp}, we have
 \[ \mathrm{Sym}\bigl(L_{2m}^{(1)}\bigr) = \dil{\frac{1}{m}} \HGPS{m}{2, 1-\frac{1}{2m}}{\cdot} \]
 so
 \[ \lim_{m \to \infty} \rho(H_{2m} \square_{2m} H_{2m}) = \dil{\frac{1}{m}} \SRMS{2, 1}{\cdot} = \lim_{m \to \infty} \rho\bigl(\mathrm{Sym}\bigl(L_{2m}^{(1)}\bigr)\bigr). \]
\end{Example}
\begin{Example}[continuation of Example \ref{exa:HermProjCommutator}] \label{exa:HermProjCommLimit}
 In Example \ref{exa:HermProjCommutator}, we found that
 \[ H_{2m} \square_{2m} R_{2m}^{(m)} = \dil{\frac{1}{\sqrt{2m}}} \HGPS{m}{1, 2}{2+\frac{1}{m}} \in \P_{2m}(\rr). \]
 In \cite{NS98}, it is shown that the free commutator of a semicircular element and a projection with trace $1/2$ has a semicircular distribution with radius $\sqrt{2}$. In the limit, the formula above gives
 \begin{align*}
 \lim_{m \to \infty} \rho\bigl(H_{2m} \square_{2m} R_{2m}^{(m)}\bigr) &= \lim_{m \to \infty} \rho \Bigl( \dil{\frac{1}{\sqrt{2m}}} \HGPS{m}{1, 2}{2+\frac{1}{m}} \Bigr) = \lim_{m \to \infty} \rho \Bigl( \dil{\frac{1}{\sqrt{2m}}} \HGPS{m}{1}{\cdot} \Bigr) \\
 &= \lim_{m \to \infty} \rho \Bigl( \dil{\frac{1}{\sqrt{2m}}} \HGPS{m}{1-\frac{1}{2m}}{\cdot} \Bigr) = \lim_{m \to \infty} \rho\Bigl( \dil{\frac{1}{\sqrt{2}}} H_{2m}\Bigr),
 \end{align*}
 which is semicircular with radius $\sqrt{2}$.
\end{Example}

\begin{Example}[continuation of Example \ref{exa:LaguerreCommutator}] \label{exa:LaguCommLimit}
 In Example \ref{exa:LaguerreCommutator}, we found that
 \[ L_{2m}^{(\lambda)} \square_{2m} L_{2m}^{(\mu)} = \dil{\frac{1}{2m^2}} \HGPS{m}{2\lambda, 2\mu, 2, 2, 1-\frac{1}{2m}}{2+\frac{1}{m}} \in \P_{2m}(\rr). \]
 Asymptotically, we then have
 \begin{align*}
 \lim_{m \to \infty} \rho\bigl(L_n^{(\lambda)} \square_{2m} L_n^{(\mu)}\bigr) &= \lim_{m \to \infty} \rho \Bigl( \dil{\frac{1}{2m^2}} \HGPS{m}{2\lambda, 2\mu, 2, 2, 1-\frac{1}{2m}}{2+\frac{1}{m}} \Bigr) \\
 &= \lim_{m \to \infty} \rho \left( \dil{\frac{1}{2m^2}} \HGPS{m}{2\lambda, 2\mu, 2, 1}{\cdot} \right).
 \end{align*}
 The free commutator of free Poisson variables with parameters $\lambda$ and $\mu$ is not explicitly computed in \cite{NS98}, but the pieces are all there: the $S$-transform of the square of this free commutator comes out to
 \[ \frac{4}{(z+2\lambda)(z+2\mu)(z+2)(z+1)}, \]
 which matches the limiting $S$-transform above.
\end{Example}

\begin{Example}[continuation of Example \ref{exa:ProjBernComm}] \label{exa:BernCommLimit}
 It was established in Example~\ref{exa:ProjBernComm} that with ${B_{2m}(x) := \bigl(x^2-1\bigr)^m}$, we have
 \[ B_{2m} \square_{2m} B_{2m} = \dil{2} \HGPS{m}{1, 1}{1-\frac{1}{2m}, 2+\frac{1}{m}} \in \P_{2m}(\rr). \]
 From \cite{NS98}, it is known that with $\mu = \frac{1}{2}(\delta_1+\delta_{-1})$, one has $\mu \square \mu = \mu \boxplus \mu$. The computation of the latter \cite[Example 12.8]{NS-book} is a standard example in free probability: $\mu \boxplus \mu$ is identified as the so-called \emph{arcsine} distribution.

 We can test our finite result by taking the limit of empirical root distributions:
 \[ \lim_{m \to \infty} \rho(B_{2m} \square_{2m} B_{2m}) = \dil{2} \SRMS{1}{2} = \lim_{m \to \infty} \rho(\mathrm{Sym}(B_{2m})). \]
 Since $B_{2m}$ is even, $\mathrm{Sym}(B_{2m}) = B_{2m} \boxplus_{2m} B_{2m}$, so we recover the corresponding result from free probability.
\end{Example}

\begin{Example}[continuation of Example \ref{exa:ProjBernComm}]
 \label{exa:ProjCommLimit}
 Recall from Example \ref{exa:ProjBernComm} that
 \[ R_{2m}^{(m)} \square_{2m} R_{2m}^{(m)} = \dil{\frac{1}{2}} \HGPS{m}{1, 1}{1-\frac{1}{2m}, 2+\frac{1}{m}} .\]
 In \cite{NS98}, it is shown that the commutator of free projections with trace $\frac{1}{2}$ has the arcsine distribution on the interval $\bigl[-\frac{1}{2},\frac{1}{2}\bigr]$. This distribution may be described as a scaled version of the free additive convolution of the measure $\frac{1}{2}(\delta_{-1}+\delta_1)$ with itself. From the formula above, we have
 \begin{align*}
 \lim_{m \to \infty} \rho\bigl( R_{2m}^{(m)} \square_{2m} R_{2m}^{(m)}\bigr) &= \lim_{m \to \infty} \rho \left( \dil{\frac{1}{2}} \HGPS{m}{1, 1}{1-\frac{1}{2m}, 2+\frac{1}{m}} \right) = \dil{\frac{1}{2}} \SRMS{1}{2}.
 \end{align*}
 Since this is the arcsine distribution, we recover the result from \cite{NS98}.
\end{Example}

\section[Summary of results in iFj notation]{Summary of results in $\boldsymbol{{}_iF_j}$ notation}\label{sec:results.standard.notation}

In this section, we compile the main results for even hypergeometric polynomials from Section~\ref{sec:EvenHyper}. The difference is that now we use the standard notation coming from hypergeometric functions.

\begin{Notation}[even hypergeometric polynomials]\label{not:even.hgp.new}
 Given $i,j,m\in \nn$, $\bm a =(a_1, \dots, a_i) \in \rr^i$ and~${\bm b =(b_1, \dots, b_j)\in \rr^j}$, an even hypergeometric polynomial is a polynomial of the form
\[\HGF{i+1}{j}{-m,\bm a }{\bm b }{x^2}=\sum_{k=0}^{m}\frac{\raising{-m}{k}\raising{\bm a}{k}}{\raising{\bm b}{k}} \frac{ x^{2k}}{k!}.\]
\end{Notation}

\begin{Example}[Bernoulli and Hermite polynomials]
\label{exm:Bernoulli.Herm}
The simplest even hypergeometric polynomial is the Bernoulli polynomial from Example \ref{exm:Bernoulli}
\[ B_{2m}(x)=\HGF{1}{0}{-m}{\cdot}{x^2}=\bigl(1-x^2\bigr)^m. \]
Another important example are the Hermite polynomials
\[ H_{2m}(x) = \HGF{i+1}{j}{-m,}{\frac{1}{2}}{\frac{x^2}{m}}. \]
\end{Example}

In the following result, we adapt Theorem~\ref{thm:add.hyper.general}, which is a generalization of the last part of Theorem~\ref{thm:MFMP}.

\begin{Theorem}[additive convolution of hypergeometric polynomials]\label{thm:add.hyper.general.new}
 Let $c_1,c_2,c_3\in \rr$ be constants, and let $l_1,l_2,l_3,n\in \nn$ be numbers such that $l_k$ divides $n$ for $k=1,2,3$. Consider tuples~$\bm a_1$, $\bm a_2$, $\bm a_3$, $\bm b_1$, $\bm b_2$, $\bm b_3 $ of sizes $i_1,i_2,i_3,j_1,j_2,j_3 \in \nn$, and assume that
 \[ \HGF{j_1}{i_1}{\bm b_1}{\bm a_1}{c_1 x^{l_1} } \HGF{j_2}{i_2}{\bm b_2}{\bm a_2}{c_2 x^{l_2}}=\HGF{j_3}{i_3}{\bm b_3}{\bm a_3}{c_3x^{l_3}}. \]
 Then, if for $k=1,2,3$ we consider the polynomials
 \[ p_k(x)= \alpha_k \HGF{i_k+1}{j_k+l_k-1}{-\frac{n}{l_k},1-\frac{n}{l_k}- \bm a_k}{1-\frac{n}{l_k}- \bm b_k, \frac{1}{l_k}, \frac{2}{l_k} , \dots, \frac{l_k-1}{l_k}} { \frac{(-1)^{i_k+j_k+1}}{l_k^{l_k}c_k} x^{l_k}}, \]
 where $\alpha_k$ is just a constant to make the polynomial monic, then we get that $p_1 \boxplus_n p_2 = p_3$.
\end{Theorem}

The proof is given in Section~\ref{sec:EvenHyper}, the main idea is to use the definition of $\boxplus_n$ in terms of differential operators given in equation \eqref{thirdDefAdditiveConv}. We can use Theorem~\ref{thm:add.hyper.general.new} and Proposition~\ref{prop:EvenAdd} to compute the additive convolution of some even hypergeometric polynomials, the results are summarized in Table~\ref{tab:sum.hgp.new}, which is just a translation of Table~\ref{tab:sum.hgp}.

\begin{table}[ht]\renewcommand{\arraystretch}{1.8}
 \centering
 %\begin{adjustbox}{center}
 \begin{tabular}{|c|c|c|}
 \hline $p$ & $q$ & $p\boxplus_{2m} q$ \\
 \hline \hline
 $\HGF{2}{1}{-m, a_1}{\frac{1}{2}}{x^2}$ & $\HGF{2}{1}{-m, a_2}{\frac{1}{2}}{x^2}$ & \small $\HGF{4}{3}{-m, a_1, a_2, a_1+a_2+m}{\frac{a_1+a_2}{2}, \frac{a_1+a_2+1}{2}, \frac{1}{2}}{4x^2}$ \\
 \hline
 $\HGF{1}{1}{-m}{\frac{1}{2}}{c_1x^2}$ & $\HGF{1}{1}{-m}{\frac{1}{2}}{c_2x^2}$ & $\HGF{1}{1}{-m}{\frac{1}{2}}{(c_1+c_2)x^2}$ \\
 \hline $ \HGF{1}{2}{-m}{b_1, \frac{1}{2}}{x^2}$ & $ \HGF{1}{2}{-m}{b_2, \frac{1}{2}}{x^2}$ & $ \HGF{1}{2}{-m}{b_1+b_2, \frac{1}{2}}{x^2}$ \\
 \hline
 \small$ \HGF{1}{2}{-m}{b_1+b_2-a, \frac{1}{2}}{x^2}$ & \small $ \HGF{2}{3}{-m, a}{a-b_1+m-1, a-b_2+m-1, \frac{1}{2}}{x^2}$ & $ \HGF{1}{2}{-m, a}{b_1, b_2 , \frac{1}{2}}{x^2}$ \\
 \hline
 \small$ \HGF{2}{3}{-m, b_1+b_2-m+\frac{3}{2}}{b_1, b_2, \frac{1}{2}}{x^2}$ &\small$ \HGF{2}{3}{-m, b_1+b_2-m+\frac{3}{2}}{b_1, b_2, \frac{1}{2}}{x^2}$ &\scriptsize$ \HGF{3}{4}{-m, b_1+b_2-m+\frac{3}{2}, 2(b_1+b_2)-3m+3}{2b_1-m+1, b_1+b_2-m+1, 2b_2-m+1, \frac{1}{2}}{x^2}$\\
 \hline
 \end{tabular}
% \end{adjustbox}
 \caption{Sum of even hypergeometric polynomials.}
 \label{tab:sum.hgp.new}
\end{table}

A special instance of Theorem~\ref{thm:add.hyper.general.new} allows us to compute the symmetrization of some hypergeometric polynomials. Recall from Section~\ref{ssec:symmetrization} that the symmetrization of $p \in \P_n$ is defined as
$\mathrm{Sym}(p):=p\boxplus_n (\dil{-1} p )$.
Then, Lemma \ref{lem:symmetrization.H} is equivalent to the following.

\begin{Lemma}\label{lem:symmetrization.H.new}
 Consider tuples $\bm a$, $\bm a'$, $\bm b$, $\bm b'$ of sizes $i$, $i'$, $j$, $j'$, and assume that
 \[
 \HGF{j}{i}{\bm b}{\bm a}{x} \HGF{j}{i}{\bm b}{\bm a}{ -x}=\HGF{j'}{i'}{ \bm b'}{\bm a'}{c x^2}.
 \]
 Then we have
 \begin{gather*}
 \operatorname{Sym}\left(\alpha_1\HGF{i+1}{j}{-2m,1-2m-\bm a }{1-2m-\bm b}{x}\right)\\
 \qquad =\alpha_2 \HGF{i+1}{j+1}{-m,1-m-\bm a' }{1-m-\bm b', \frac{1}{2}}{4c(-1)^{i'+j'+1} x^2} ,
 \end{gather*}
 where $\alpha_1$ and $\alpha_2$ are just constants to make the polynomials monic.
\end{Lemma}
Using Lemma \ref{lem:symmetrization.H.new} and some well-known results on products of hypergeometric functions from Proposition~\ref{prop:prod.hg.series.sym}, we can compute the symmetrizations of various classical polynomials, see Table~\ref{tab:sym.hgp.new} which is just an adaptation of Table~\ref{tab:sym.hgp}.

\begin{table}[ht]\renewcommand{\arraystretch}{1.8}
 \centering
% \begin{adjustbox}{center}
 \begin{tabular}{|c|c|c|}
 \hline
 Polynomial & $p$ & $\mathrm{Sym}(p)$ \\
 \hline
 Laguerre & $\HGF{1}{1}{-2m}{ b}{x}$ & $\HGF{1}{2}{-m}{ b+m, \frac{1}{2}}{4x^2}$ \\
 \hline
 Bessel & $\HGF{2}{0}{-2m,a}{\cdot}{x}$ & $\HGF{4}{1}{-m,a+m,\frac{a+1}{2},\frac{a}{2}}{ \frac{1}{2}}{-x^2}$ \\
 \hline
 Jacobi & $\HGF{2}{0}{-2m,a}{b}{x}$ & $\HGF{4}{3}{-m,a+m,\frac{a+1}{2},\frac{a}{2}}{b+m, a-b-m+1, \frac{1}{2}}{x^2}$ \\
 \hline
 \small Lag $\boxtimes_{2m}$ Lag & $\HGF{1}{2}{-2m}{b,d}{x}$ & $\HGF{2}{5}{b+d+3m-1}{b+m,d+m,\frac{b+d}{2}+m,\frac{b+d-1}{2}+m,\frac{1}{2}}{16x^2}$ \\
 \hline
 \small Bes $\boxtimes_{2m}$ Bes & $\HGF{1}{2}{-2m,a,c}{\cdot}{x}$ & \small $\HGF{8}{4}{-m,a+m,c+m, \frac{a+1}{2},\frac{c+1}{2},\frac{a}{2},\frac{c}{2},\frac{a+c+1}{2}+m}{\frac{a+c+m+2}{3},\frac{a+c+m+1}{3},\frac{a+c+m}{3},\frac{1}{2}}{\frac{-64x^2}{27}}$ \\
 \hline
 \end{tabular}
 \caption{Symmetrization of hypergeometric polynomials.} \label{tab:sym.hgp.new}
\end{table}

Finally, Proposition~\ref{prop:EvenHyperMultConv} can be rewritten as follows.

\begin{Proposition}[multiplicative convolution of even hypergeometric polynomials] \label{prop:EvenHyperMultConv.new}
 Consider tuples $\bm a_1$, $\bm a_2$, $\bm b_1$, $\bm b_2$ of sizes $i_1$, $i_2$, $j_1$, $j_2$. Then
 \begin{gather*}
 \alpha_1 \HGF{i_1+1}{j_1}{-m, \bm a_1}{\bm b_1 }{x^2}\boxtimes_{2m}\alpha_2 \HGF{i_2+1}{j_2}{-m, \bm a_2}{\bm b_2 }{x^2} \\
 \qquad= \alpha_3 \HGF{i_1+i_2+2}{j_1+j_2+1}{-m, \frac{1}{2}, \bm a_1 , \bm a_2 }{\frac{1}{2}-m, \bm b_1, \bm b_2}{x^2} , \end{gather*}
 where $\alpha_1$, $\alpha_2$ and $\alpha_3$ are just constants to make the polynomials monic.
\end{Proposition}

This identity is specially useful to compute the commutator of hypergeometric polynomials. Recall from Section~\ref{sec:commutator} that for polynomials $p$ and $q$ with degree $2m$, their commutator is defined as the polynomial
\[
p \square_{2m} q := \mathrm{Sym}(p)\boxtimes_{2m} \mathrm{Sym}(q) \boxtimes_{2m} z_{2m},
\]
where
\[ z_{2m}(x) =\alpha \HGF{4}{3}{-m, \frac{1}{2}-m, \frac{1}{2}-m, m+2}{m+1, m+1, \frac{1}{2}}{\frac{x^2}{4}} , \]
and $\alpha$ is just a constant to make the polynomial monic.

In Table~\ref{tab:com.hgp.new}, we provide examples of the commutator of hypergeometric polynomials directly in terms of ${}_iF_j$ functions. These are the examples from Table~\ref{tab:com.hgp} after changing the notation.

\begin{table}[ht]\renewcommand{\arraystretch}{1.8}
 \centering
 \begin{tabular}{|c|c|c|c|}
 \hline $p$ & $q$ & $p \square_{2m} q$ \\
 \hline \hline
 $\HGF{2}{1}{-2m, 1 }{-m+1}{x}$ & $\HGF{2}{1}{-2m, 1 }{-m+1}{x}$ & $\HGF{3}{2}{-m, \frac{1}{2}, m+2 }{1, 1}{\frac{1}{4}x^2}$ \\
 \hline
 $\HGF{1}{1}{-m }{\frac{1}{2}}{x^2}$ & $\HGF{1}{1}{-m }{\frac{1}{2}}{x^2}$ & $\HGF{2}{3}{-m, m+2 }{\frac{1}{2}, m+1, m+1}{\frac{1}{4}x^2}$ \\
 \hline
 $\HGF{1}{1}{-m }{\frac{1}{2}}{x^2}$ & $\HGF{2}{1}{-2m, 1 }{-m+1}{x}$ & $\HGF{2}{2}{-m, m+2 }{1, m+1}{\frac{1}{4}x^2}$ \\
 \hline
 $\HGF{1}{1}{-2m }{b}{x}$ & $\HGF{1}{1}{-2m}{d}{x}$ & $\HGF{2}{5}{-m, m+2}{2b-m, 2d-m, m+1, m+1 , \frac{1}{2}}{\frac{1}{4}x^2}$ \\
 \hline
 $\HGF{2}{1}{-2m, 2m+1}{1}{x}$ & $\HGF{2}{1}{-2m, 2m+1}{1}{x}$ &\small $\HGF{6}{5}{-m, 3m+1, 3m+1, m+\frac{1}{2}, m+\frac{1}{2}, m+2}{m+1, m+1, m+1, m+1, \frac{1}{2}}{\frac{1}{4}x^2}$ \\
 \hline
 \end{tabular}
 \caption{Examples of finite free commutators.} \label{tab:com.hgp.new}
\end{table}

\subsection*{Acknowledgements}

This project originated during the 2023 Workshop in Analysis and Probability at Texas A\&M University. We thank the organizers of the conferences YMC*A and IWOTA in August 2024, where two of the authors had fruitful discussions. We also thank Andrew Campbell for bringing some recent references to our attention. We greatly appreciate the corrections and useful suggestions offered by the anonymous referees.
D.P. was partially supported by the AMS-Simons Travel Grant. The author also appreciates the hospitality of University of Virginia during April~2024.

\pdfbookmark[1]{References}{ref}
\LastPageEnding


\begin{thebibliography}{99}
\footnotesize\itemsep=0pt

\bibitem{arizmendi2024s}
Arizmendi O., Fujie K., Perales D., Ueda Y., {$S$}-transform in finite free
 probability, \href{http://arxiv.org/abs/2408.09337}{arXiv:2408.09337}.

\bibitem{arizmendi2021finite}
Arizmendi O., Garza-Vargas J., Perales D., Finite free cumulants:
 multiplicative convolutions, genus expansion and infinitesimal distributions,
 \href{https://doi.org/10.1090/tran/8884}{\textit{Trans. Amer. Math. Soc.}}
 \textbf{376} (2023), 4383--4420,
 \href{http://arxiv.org/abs/2108.08489}{arXiv:2108.08489}.

\bibitem{AHS13}
Arizmendi O., Hasebe T., Sakuma N., On the law of free subordinators,
 \textit{ALEA Lat. Am.~J. Probab. Math. Stat.} \textbf{10} (2013), 271--291,
 \href{http://arxiv.org/abs/1201.0311}{arXiv:1201.0311}.

\bibitem{arizmendi2018cumulants}
Arizmendi O., Perales D., Cumulants for finite free convolution,
 \href{https://doi.org/10.1016/j.jcta.2017.11.012}{\textit{J.~Combin. Theory
 Ser.~A}} \textbf{155} (2018), 244--266,
 \href{http://arxiv.org/abs/1611.06598}{arXiv:1611.06598}.
 
\bibitem{arizmendi2009S}
Arizmendi O., P\'erez-Abreu V., The {$S$}-transform of symmetric probability measures
 with unbounded supports,
 \href{https://doi.org/10.1090/S0002-9939-09-09841-4}{\textit{Proc. Amer.
 Math. Soc.}} \textbf{137} (2009), 3057--3066.

\bibitem{benaych2009rectangular}
Benaych-Georges F., Rectangular random matrices, related convolution,
 \href{https://doi.org/10.1007/s00440-008-0152-z}{\textit{Probab. Theory
 Related Fields}} \textbf{144} (2009), 471--515,
 \href{http://arxiv.org/abs/math.OA/0507336}{arXiv:math.OA/0507336}.

\bibitem{Billingsley-convergence}
Billingsley P., Convergence of probability measures, 2nd~ed., \textit{Wiley Ser. Probab. Statist. Appl. Probab. Statist.},
 \href{https://doi.org/10.1002/9780470316962}{John Wiley \& Sons}, New York,
 1999.

\bibitem{campbell2024free}
Campbell A., Free infinite divisibility, fractional convolution powers, and
 {A}ppell polynomials,
 \href{http://arxiv.org/abs/2412.20488}{arXiv:2412.20488}.

\bibitem{campbell2024universality}
Campbell A., O'Rourke S., Renfrew D., Universality for roots of derivatives of
 entire functions via finite free probability,
 \href{http://arxiv.org/abs/2410.06403}{arXiv:2410.06403}.

\bibitem{campbell2022commutators}
Campbell J., Commutators in finite free probability~{I},
 \href{http://arxiv.org/abs/2209.00523}{arXiv:2209.00523}.

\bibitem{DJ02}
Driver K., Jordaan K., Zeros of {${}_3F_2\bigl(\begin{smallmatrix} -n,b,c\\
 d,e\end{smallmatrix};z\bigr)$} polynomials,
 \href{https://doi.org/10.1023/A:1020126822435}{\textit{Numer. Algorithms}}
 \textbf{30} (2002), 323--333.

\bibitem{TalkGribinski}
Gribinski A., Rectangular finite free probability theory, {O}nline talk at {UC}
 {B}erkeley {P}robabilistic {O}perator {A}lgebra {S}eminar, April~24, 2023.

\bibitem{gribinski2022rectangular}
Gribinski A., Marcus A.W., A rectangular additive convolution for polynomials,
 \href{https://doi.org/10.5070/c62156888}{\textit{Comb. Theory}} \textbf{2}
 (2022), 16, 42~pages, \href{http://arxiv.org/abs/1904.11552}{arXiv:1904.11552}.

\bibitem{grinshpan2013generalized}
Grinshpan A.Z., Generalized hypergeometric functions: product identities and
 weighted norm inequalities,
 \href{https://doi.org/10.1007/s11139-013-9487-x}{\textit{Ramanujan~J.}}
 \textbf{31} (2013), 53--66.

\bibitem{hinz2010multiplicative}
Hinz M., M{\l}otkowski W., Multiplicative free square of the free {P}oisson
 measure and examples of free symmetrization,
 \href{https://doi.org/10.4064/cm119-1-8}{\textit{Colloq. Math.}} \textbf{119}
 (2010), 127--136.

\bibitem{MR2656096}
Koekoek R., Lesky P.A., Swarttouw R.F., Hypergeometric orthogonal polynomials
 and their {$q$}-analogues, \textit{Springer Monogr. Math.},
 \href{https://doi.org/10.1007/978-3-642-05014-5}{Springer}, Berlin, 2010.

\bibitem{KM16}
Kornyik M., Michaletzky G., Wigner matrices, the moments of roots of {H}ermite
 polynomials and the semicircle law,
 \href{https://doi.org/10.1016/j.jat.2016.07.006}{\textit{J.~Approx. Theory}}
 \textbf{211} (2016), 29--41,
 \href{http://arxiv.org/abs/1512.03724}{arXiv:1512.03724}.

\bibitem{marcus}
Marcus A.W., Polynomial convolutions and (finite) free probability,
 \href{http://arxiv.org/abs/2108.07054}{arXiv:2108.07054}.

\bibitem{MSS15}
Marcus A.W., Spielman D.A., Srivastava N., Finite free convolutions of
 polynomials,
 \href{https://doi.org/10.1007/s00440-021-01105-w}{\textit{Probab. Theory
 Related Fields}} \textbf{182} (2022), 807--848,
 \href{http://arxiv.org/abs/1504.00350}{arXiv:1504.00350}.

\bibitem{marden1966geometry}
Marden M., Geometry of polynomials, 2nd ed., \textit{Math. Surveys}, American
 Mathematical Society, Providence, RI, 1966.

\bibitem{martinez2025zeros}
Mart\'{\i}nez-Finkelshtein A., Morales R., Perales D., Zeros of generalized
 hypergeometric polynomials via finite free convolution: applications to
 multiple orthogonality,
 \href{https://doi.org/10.1007/s00365-025-09703-w}{\textit{Constr. Approx.}},
 {t}o appear, \href{http://arxiv.org/abs/2404.11479}{arXiv:2404.11479}.

\bibitem{mfmp2024hypergeometric}
Mart\'{\i}nez-Finkelshtein A., Morales R., Perales D., Real roots of
 hypergeometric polynomials via finite free convolution,
 \href{https://doi.org/10.1093/imrn/rnae120}{\textit{Int. Math. Res. Not.}}
 \textbf{2024} (2024), 11642--11687,
 \href{http://arxiv.org/abs/2309.10970}{arXiv:2309.10970}.

\bibitem{MS-book}
Mingo J.A., Speicher R., Free probability and random matrices, \textit{Fields
 Inst. Monogr.}, Vol.~35,
 \href{https://doi.org/10.1007/978-1-4939-6942-5}{Springer}, New York, 2017.

\bibitem{mlotkowski2011symmetrization}
M{\l}otkowski W., Sakuma N., Symmetrization of probability measures,
 pushforward of order~2 and the {B}oolean convolution, in Noncommutative
 {H}armonic {A}nalysis with {A}pplications to {P}robability~{III},
 \textit{Banach Center Publ.}, Vol.~96,
 \href{https://doi.org/10.4064/bc96-0-18}{Institute of Mathematics of the
 Polish Academy of Sciences}, Warsaw, 2012, 271--276.

\bibitem{NS98}
Nica A., Speicher R., Commutators of free random variables,
 \href{https://doi.org/10.1215/S0012-7094-98-09216-X}{\textit{Duke Math.~J.}}
 \textbf{92} (1998), 553--592,
 \href{http://arxiv.org/abs/funct-an/9612001}{arXiv:funct-an/9612001}.

\bibitem{NS-book}
Nica A., Speicher R., Lectures on the combinatorics of free probability,
 \textit{London Math. Soc. Lecture Note Ser.}, Vol.~335,
 \href{https://doi.org/10.1017/CBO9780511735127}{Cambridge University Press},
 Cambridge, 2006.

\bibitem{MR2723248}
Olver F.W.J., Lozier D.W., Boisvert R.F., Clark C.W. (Editors), {NIST} handbook
 of mathematical functions, Cambridge University Press, Cambridge, 2010.

\bibitem{perez2012free}
P\'erez-Abreu V., Sakuma N., Free infinite divisibility of free multiplicative
 mixtures of the {W}igner distribution,
 \href{https://doi.org/10.1007/s10959-010-0288-5}{\textit{J.~Theoret.
 Probab.}} \textbf{25} (2012), 100--121,
 \href{http://arxiv.org/abs/0910.1199}{arXiv:0910.1199}.

\bibitem{suffridge1976starlike}
Suffridge T.J., Starlike functions as limits of polynomials, in Advances in
 {C}omplex {F}unction {T}heory, \textit{Lecture Notes in Math.}, Vol.~505,
 \href{https://doi.org/10.1007/BFb0081105}{Springer}, Berlin, 1976, 164--203.

\bibitem{szego1922bemerkungen}
Szeg\"o G., Bemerkungen zu einem {S}atz von~{J}.{H}.~{G}race \"uber die
 {W}urzeln algebraischer {G}leichungen,
 \href{https://doi.org/10.1007/BF01485280}{\textit{Math.~Z.}} \textbf{13}
 (1922), 28--55.

\bibitem{voiculescu1992free}
Voiculescu D.V., Dykema K.J., Nica A., Free random variables, \textit{CRM
 Monogr. Ser.}, Vol.~1, \href{https://doi.org/10.1090/crmm/001}{American
 Mathematical Society}, Providence, RI, 1992.

\bibitem{walsh1922location}
Walsh J.L., On the location of the roots of certain types of polynomials,
 \href{https://doi.org/10.2307/1989023}{\textit{Trans. Amer. Math. Soc.}}
 \textbf{24} (1922), 163--180.

\end{thebibliography}
\end{document}